\documentclass[11pt,twoside]{amsart}

\usepackage{amsfonts}
\usepackage{hyperref}
\usepackage{etoolbox}
\usepackage{amsmath}
\usepackage{amssymb}
\usepackage{amsthm}
\usepackage{dsfont}
\usepackage{pifont}
\usepackage{mathrsfs}
\usepackage{graphicx}
\usepackage{enumerate}
\usepackage[all]{xy}
\usepackage{geometry}
\usepackage{amsfonts}
\usepackage{fancyhdr}
\usepackage{tikz}
\usetikzlibrary{arrows}
\usepackage{lscape}
\usetikzlibrary{tikzmark,calc}

\usepackage{amsaddr}

\usepackage[backend=biber,style=alphabetic,sorting=nty,maxnames=5]{biblatex}
\newcommand{\Aa}{\mathcal{A}}
\newcommand{\Bb}{\mathcal{B}}
\newcommand{\Cc}{\mathcal{C}}
\newcommand{\Dd}{\mathcal{D}}
\newcommand{\Ee}{\mathcal{E}}

\newcommand{\Gg}{\mathcal{G}}
\newcommand{\Hh}{\mathcal{H}}

\newcommand{\Xx}{\mathcal{X}}

\newcommand{\C}{\mathbb{C}}
\newcommand{\F}{\mathbb{F}}

\newcommand{\Q}{\mathbb{Q}}

\newcommand{\Z}{\mathbb{Z}}

\DeclareMathOperator\Tor{Tor}
\DeclareMathOperator\ho{Hom}

\DeclareMathOperator\GL{GL}

\DeclareMathOperator\Ob{Ob}

\DeclareMathOperator\md{md}
\DeclareMathOperator\lcm{lcm}

\newcommand{\Set}{\mathrm{\textbf{Set}}}

\newcommand{\modu}{\mathrm{\textbf{mod}}}

\newcommand{\Ab}{\mathrm{\textbf{Ab}}}

\newcommand{\bigslant}[2]{{\raisebox{.2em}{$#1$}\left/\raisebox{-.2em}{$#2$}\right.}}
\newcommand{\xdownarrow}[1]{%
  {\left\downarrow\vbox to #1{}\right.\kern-\nulldelimiterspace}
}

\newcommand{\infl}{\ar@{{(}->}}
\newcommand{\defl}{\ar@{->>}}

\newcommand{\intv}[1]{[\![#1]\!]}
\newcommand{\muls}[1]{ \left\{\kern-0.6em \left\{ #1\right\}\kern-0.6em \right\} }

\newcommand{\epsi}{\varepsilon}

\newcommand{\spa}{\vspace*{1ex}}

\newcommand{\hhat}[1]{\widehat{#1}}

\newcommand{\nit}[1]{\textbf{\emph{#1}}}

\theoremstyle{plain}
\newtheorem{prop}{Proposition}[section]
\newtheorem{prop-def}[prop]{Proposition-Definition}
\newtheorem{lem}[prop]{Lemma}
\newtheorem{theo}[prop]{Theorem}
\newtheorem{cor}[prop]{Corollary}
\newtheorem{definition}[prop]{Definition}

\newtheorem*{theo*}{Theorem}

\theoremstyle{remark}
\newtheorem{rem}[prop]{Remark}
\newtheorem{exemple}[prop]{Example}

 \normalbaselines
\title[Dehornoy-Lafont order complex for categories]{Generalization of the Dehornoy-Lafont order complex to categories. Application to exceptional braid groups}

\author{Owen Garnier}
\address{LAMFA, Université de Picardie Jules Verne, CNRS UMR 7352,\\ 33, rue Saint-Leu, 80000, Amiens, France.}
\email{o.garnier@u-picardie.fr}
\date{\today}
\subjclass[2020]{Primary 20J06 and 18G35 ; Secondary 20F36 and 20F55}
\keywords{Garside Category, Gaussian Category, Homology, Complex braid groups}

\begin{document}

\begin{abstract}
The homology of a Garside monoid, thus of a Garside group, can be computed efficiently through the use of the order complex defined by Dehornoy and Lafont. We construct a categorical generalization of this complex and we give some computational techniques which are useful for reducing computing time.

We then use this construction to complete results of Salvetti, Callegaro and Marin regarding the homology of exceptional complex braid groups. We most notably study the case of the Borchardt braid group $B(G_{31})$ through its associated Garside category.
\end{abstract}

\maketitle
\tableofcontents

\section*{Introduction}


The question of finding suitable ways to compute the homology of specific classes of groups (in particular by finding computationally efficient resolutions of the trivial module) is a broad field of study. The particular case of the classical braid group (as defined by Artin in \cite{artin}) was first studied by Arnold in \cite{arn2} using methods of algebraid topology (more precisely Alexander duality). His work was then adapted to the more general case of braid groups associated to Coxeter groups (i.e spherical Artin groups). This is summarized in the survey \cite{vershinin}.


On the other hand, the combinatorial behavior of the classical braid group was first studied by F. Garside in \cite{garthesis}, in which the author gives a solution to both the word problem and the conjugacy problem. These methods were later generalized to other spherical Artin groups \cite{brieskornsaito}, \cite{deligneimmeuble}, \cite{bestvina}, \cite{adjan}, \cite{elrifaimorton}, and then to the class of \emph{Garside groups} \cite{dehpar}, \cite{picantincentre}, \cite{centralizergarside}, \cite{godellepara}, \cite{rigidity}, into what is now called Garside theory. This theory was notably proved in \cite{dehlaf} and \cite{cmw} to provide convenient resolutions for homology computations.

Moreover, before the coining of the term Garside group by Dehornoy and Paris, the combinatorial properties of Artin groups had already been used in \cite{salvetti} and \cite{squier} to provide resolutions and homology computations. The approach of the latter was then generalized in \cite[Section 4]{dehlaf} to a class of \emph{Gaussian monoids}, which encompasses Garside monoids.

Afterwards, near the end of the 2000's, Garside theory began to be further developed into the theory of Garside categories  \cite{kgarcat}, \cite{besgar}. This culminated in the publication in 2015 of the reference book \cite{ddgkm}, which summarizes the state of the art and the adaptations of Garside theory to a categorical context. Among these adaptations, the first resolution defined in \cite{dehlaf} is adapted to categories. However, the second resolution -the \emph{order complex}- although briefly mentioned, is not explicitly adapted to the case of a category. Furthermore, the constructions of free modules over a category and of the homology of categories are not detailed in depths. 

In this paper we directly address these two problems. We first show, following the arguments in \cite{squier}, that the homology of a category coincides with that of its enveloping groupoid under suitable assumptions. As we actually plan to compute group homology, we also state that the homology of a group is the same as the homology of a groupoid to which it is equivalent. We obtain the following theorem. For the concept of a left-Ore category, see Definition \ref{leftore}.

\begin{theo*}(Corollary \ref{cor2.16} and Proposition \ref{groupoidgroup})\newline
Let $\Cc$ be a left-Ore category, and let $\Gg$ be the groupoid of fractions of $\Cc$. For every $\Cc$-module $M$ we have $H_*(\Cc,M)=H_*(\Gg,\Z\Gg\otimes_\Cc M)$. Furthermore, if $G$ is a group, then every equivalence of categories $G\to \Gg$ induces a $G$-module structure $M'$ on $M$, and we have $H_*(\Cc,M)=H_*(G,M')$.
\end{theo*}

We then construct a categorical analogue of the order complex of Dehornoy and Lafont. To do this we introduce a notion of (locally) Gaussian category, which generalizes the notion of Garside category, as considered in \cite{besgar}. The definitions are direct adaptations of \cite[Section 4]{dehlaf} to the categorical context. The main difference is that we now keep track of the sources and targets of morphisms and of the cells. Such a point of view is implicitly present in the original article of Dehornoy and Lafont, where some diagrams already represent elements of a monoid by arrows, suggesting a categorical viewpoint on these constructions. We get the same result as in \cite{dehlaf}: the complex we obtain is a free resolution of the trivial module over a locally left-Gaussian category (Proposition \ref{2.15}). As the definition of the order complex is dependent on some additional choice, we also give a practical method for handling this choice, in order to reduce the number of cells. 

Since Garside categories are a particular case of Gaussian categories, this theoretically applies to a vast family of examples. For instance, finite index subgroups of Garside groups and centralizer of periodic elements in Garside groups are naturally associated to some Garside category (which is not a Garside monoid a priori). However, the size of these structure, along with the recursive nature of the differential of the order complex, makes it really hard to perform actual computations. This is where the method we propose for reducing the number of cells may be used most effectively.

As a concrete application of our construction, we compute the homology of complex braid groups with coefficients in various modules. Recall from \cite[Section 2]{bmr} that a complex braid group $B(W)$ is defined as the fundamental group of the space of regular orbits of some complex reflection group $W\leqslant \GL_n(\C)$. We can restrict our attention to irreducible groups without loss of generality. The classification of irreducible complex reflection groups was done in 1954 by Shephard and Todd (see \cite{shetod}): there is an infinite series $G(de,e,n)$ with three integer parameters, and 34 exceptional cases, labeled $G_4,\ldots,G_{37}$. We restrict our attention to the braid groups of these exceptional groups.

Over time, almost all exceptional braid groups have been shown to admit a Garside group structure (see for instance \cite[Section 8]{beskpi1}, \cite[Examples 11, 12, 13]{thespicantin}, \cite[Example 1]{dehpar}). The only exception (among exceptional braid groups) is the Borchardt braid group $B(G_{31})$, which is equivalent to the groupoid of fractions of a Garside category, instead of being the group of fractions of some Garside monoid. 

The Garside monoid structures of the other exceptional braid groups have already been used in \cite{homcomp1} and \cite{homcomp2} to compute homological results. The lack of a categorical version of the order complex of Derhornoy and Lafont proved to be a problem in computing the homology of the Borchardt braid group, as the only resolution available at that time afforded far too many cells for most explicit computations: only the integral homology of the Borchardt braid group could be computed in \cite{homcomp1}, using other tools like Betti numbers and reduction modulo $p$. Our adaptation of the order complex to categories allows us to complete results from \cite{homcomp1} and \cite{homcomp2} regarding this group. Furthermore, we also apply our methods of reducing the number of cells to other exceptional groups in order to give comparisons with the results in \cite{homcomp2}, where the author used the order complex defined for Garside monoids.

The paper is organized as follows. In Section 1, we review the notion of module over a category, and, in particular, the construction of a free module over a category. We also prove that the homology of a category satisfying the (left-)Ore condition coincides with the homology of its enveloping groupoid, and that the homology of a groupoid is the same as the homology of a group to which it is equivalent. In Section 2, we introduce the notion of a Gaussian category, and our generalization of the order complex to these categories, which provides an explicit free and finite resolution of the trivial module. We also give a practical procedure for reducing computations in practice. In Section 3, we apply this construction to various Garside monoids and categories associated to exceptional complex braid groups. In particular we compute the homology of the Borchardt braid groups with coefficients in the set of Laurent polynomials over the rationals. We also apply our method to reduce computations to give comparison with earlier results of \cite{homcomp1} and \cite{homcomp2}. Lastly, as we encounter the same computational problems as in \cite{homcomp2} regarding the computation of the homology of $B(G_{34})$ with coefficient in the set of Laurent polynomials over the rationals, we compute some results in finite fields in order to obtain a conjecture on this particular homology.

%

\addtocontents{toc}{\protect\setcounter{tocdepth}{1}}

\subsection*{Acknowledgments.} The computational results for the braid groups of $G_{24},G_{27},G_{29},G_{31},$ $G_{33}$ and $G_{34}$ were obtained using the MatriCS platform of the Université de Picardie Jules Verne in Amiens, France. I thank Étienne Piskorski, Laurent Renault and Jean-Baptiste Hoock for their help in using it. I also thank my PhD thesis advisor Ivan Marin for his precious help.

\addtocontents{toc}{\protect\setcounter{tocdepth}{4}}

\section{Homology of a Category}\label{homcat}
In order to define an analogue of the Dehornoy-Lafont order complex for categories, one first needs to define the homology of a category in general. In particular we need notions of resolutions, free modules over a category and tensor product of modules over a category. Then we need, just as for monoids, to give a relation between the homology of a category and the homology of its enveloping groupoid under suitable assumptions.

In accordance with the convention of \cite{ddgkm} regarding composition of arrows, the composition of the diagram 
\[\xymatrix{x\ar[r]^-f & y\ar[r]^-g & z}\] 
will be denoted $fg$. In the same vein, the categories that we will consider are assumed to be small categories. If $\Cc$ denotes a category, then the set of objects of $\Cc$ will be denoted by $\Ob(\Cc)$, and the set of morphisms from an object $x$ to an object $y$ will be denoted by $\Cc(x,y)$.

\subsection{Modules over a category}\label{modcat}
The representation theory of categories appears for instance in the study of quivers and correspondence functors (see \cite[Section 2]{boucthevenaz1} and \cite{webb}). For convenience, we provide here a definition of modules and free modules over a category.

\begin{definition} Let $\Cc$ be a category. A $\Z\Cc$-\nit{module} (or $\Cc$-module) is a contravariant functor from $\Cc$ to the category $\Ab$ of abelian groups. Equivalently, a $\Z\Cc$-module is given by a contravariant additive functor $\Z\Cc\to \Ab$.
\end{definition}

Let $M$ be a $\Z\Cc$-module. Every $x\in \Ob(\Cc)$ is mapped to an abelian group $M_x$. A morphism $f:x\to y$ induces a morphism $M(f):M_y\to M_x$. For $a\in M_y$, we denote $f.a$ instead of $M(f)(a)$. This element lies in $M_x$.

Let $g\in \Cc(y,z)$. By definition of a contravariant functor, and because of our convention for composition of arrows in $\Cc$ (which is different from the composition in $\Ab$), we have $M(fg)=M(f)\circ M(g)$. In our notation for $M(f)$, we get
\[\forall a\in M_z,~~f.(g.a)=(fg).a.\]

\begin{exemple}
One can always consider $\Z$ as a trivial $\Cc$-module by considering the functor mapping every object to $\Z$ and every morphism to the identity in $\Z$.
\end{exemple}

\begin{rem}
Another point of view, usually adopted in the representation theory of quivers, is to consider the algebra $A$ generated as a ring by all morphisms in $\Cc$, and with relations
\[fg=\begin{cases} f\circ g&\text{if the source of $g$ is the target of $f$}\\ 0 &\text{otherwise}\end{cases}\]
This point of view is in fact equivalent to ours. Indeed, considering a $\Cc$-module $M$, the abelian group $\bigoplus_{x\in \Ob(\Cc)} M_x$ naturally comes equipped with a $A$ module structure.
\end{rem}

\begin{prop-def}
We denote by $\Z\Cc-\modu$ the category of $\Z\Cc$-modules, where the morphisms are the natural transformations between functors (if $M,N$ are $\Cc$-modules, we denote by $\ho_\Cc(M,N)$ the corresponding set of morphisms in $\Z\Cc-\modu$). Since $\Ab$ is an abelian category, the category $\Z\Cc-\modu$ is an abelian category.
\end{prop-def}

As we want to consider free resolutions in the category $\Z\Cc-\modu$, we first define a notion of free module over $\Cc$. Consider the category $\Set^{\Ob(\Cc)}$ of families of sets indexed by objects of $\Cc$. The morphisms between two families $\{S_x\}_{x\in \Ob(\Cc)}$ and $\{T_x\}_{x\in \Ob(\Cc)}$ are given by families of (set-theoretic) maps $\{\varphi_x:S_x\to T_x\}_{x\in \Ob(\Cc)}$.

The category $\Z\Cc-\modu$ is endowed with a forgetful functor to $\Set^{\Ob(\Cc)}$, sending a functor $M$ to the family $\{M_x\}_{x\in \Ob(\Cc)}$. We construct a free functor, adjoint to this forgetful functor.

Let $S:=\{S_x\}_{x\in \Ob(\Cc)}$ be a family of sets. For $x\in \Ob(\Cc)$, we define $F(S)_x$ as the free abelian group over the set
\[(\Cc S)_x:=\{(g, s)~|~ g:x\to y \text{~and~} s\in S_y\}.\]
Then, for a morphism $f:x\to y$, the morphism $F(S)(f):F(S)_y\to F(S)_x$ is defined on $(\Cc S)_y$ by 
\[(g,s)\mapsto (fg,s)\in (\Cc S)_x\]
and then extended to $F(S)_y$ by linearity.

If $\varphi=\{\varphi_x\}_{x\in \Ob(\Cc)}:\{S_x\}_{x\in \Ob(\Cc)}\to \{T_x\}_{x\in \Ob(\Cc)}$ is a morphism in $\Set^{\Ob(\Cc)}$. We get maps
\[\begin{array}{cccc}\Cc \varphi_x:&(\Cc S)_x&\longrightarrow &(\Cc T)_x\\ & (g,s)&\longmapsto & (g,\varphi(s))\end{array}\]
which induce a morphism $F(\varphi)$ of $\Z\Cc$-modules between $F(S)$ and $F(T)$.

\begin{lem}
The functor $F:\Set^{\Ob(\Cc)}\to \Z\Cc-\modu$ constructed above is left-adjoint to the forgetful functor $\Z\Cc-\modu\to \Set^{\Ob(\Cc)}$.
\end{lem}
\begin{proof}
Let $S=\{S_x\}_{x\in \Ob(\Cc)}$ be a family of sets and let $M$ be a $\Z\Cc$-module. If $\varphi:S\to \{M_x\}_{x\in \Ob(\Cc)}$ is a morphism in $\Set^{\Ob(\Cc)}$, then the formula
\[\forall x\in \Ob(\Cc), \sum gs \in F(S)_x, ~\eta_x\left(\sum g s\right)=\sum g\varphi(s)\in M_x\]
yields a natural transformation $\eta:F(S)\Rightarrow L$, which is uniquely determined by $\varphi$. Conversely, if $\eta:F(S)\Rightarrow M$ is a natural transformation, then defining $\varphi(s):=\eta_x(s)$ for $s\in S_x$ induces a morphism $\varphi:S\to \{M_x\}_{x\in \Ob(\Cc)}$ in $\Set^{\Ob(\Cc)}$.

The applications $\eta\mapsto \varphi$ and $\varphi\mapsto \eta$ are inverse bijections, which give the desired adjunction.
\end{proof}

\begin{exemple}
Let $x_0\in \Ob(\Cc)$, the hom functor $\Cc(-,x_0)$ is the free functor over the family $\{M_x\}_{x\in \Ob(\Cc)}$ where
\[M_x=\begin{cases} \varnothing& \text{if }x\neq x_0\\ \{*\} & \text{if }x=x_0\end{cases}\]
In this case, the adjunction formula can be seen as a consequence of the Yoneda Lemma.
\end{exemple}

\begin{lem}
Free modules over $\Cc$, in the sense defined above, are projective objects in the category $\Z\Cc-\modu$
\end{lem}
\begin{proof}
Let $S=\{S_x\}_{x\in \Ob(\Cc)}$ be a family of sets. Let also $\epsi:M\twoheadrightarrow N$ be an epimorphism of $\Cc$-modules, and let $\eta:F(S)\to N$ be a morphism of $\Cc$-module. We want to construct a morphism $\mu:F(S)\to M$ such that $\epsi\mu=\eta$.\newline By adjunction, the morphism $\eta$ induces, for each $x\in \Ob(\Cc)$, a map $\eta'_x:S_x\to N_x$. As $\Cc-\modu$ is a functor category, stating that $\epsi$ is an epimorphism amounts to saying that, for all $x\in \Ob(\Cc)$, $\epsi_x$ is an epimorphism of abelian groups, in particular it is onto. We can thus construct a map $\mu'_x:S_x\to M_x$ such that $\epsi'_x\circ\mu'_x=\eta'_x$. Considering the morphism $\mu$ induced by $\mu'$ in the adjunction gives the desired result.
\end{proof}

\begin{rem}
In Section \ref{sec:2.1}, we will construct a free resolution of the trivial module over a Gaussian category $\Cc$. This last lemma shows that such a free resolution is indeed a projective resolution suitable for homology computations.
\end{rem}

\subsection{Tensor product}
In order to define the homology of a category $\Cc$ with coefficients in a $\Cc$-module $M$, we need to define a tensor product which extends the usual tensor product of modules. This definition is an immediate generalization of the construction given in \cite[Section 6]{mitc} (which only considers $\Z,\Cc$-bimodules and $\Cc,\Z$-bimodules).

Let $\Cc,\Dd,\Ee$ be categories. A $\Cc,\Dd$-\nit{bimodule} is a functor $\Cc\times \Dd^{op}\to \Ab$. The morphisms of bimodules between two $\Cc,\Dd$-bimodules $M$ and $N$ will be denoted by $\ho_{\Cc,\Dd}(M,N)$.

Let $M$ and $N$ be a $\Cc,\Dd$-bimodule and a $\Dd,\Ee$ bimodule, respectively, and let $(x,z)\in \Ob(\Cc\times \Ee^{op})$. We define
\[P_{x,y}:=\bigoplus_{z\in \Ob(\Dd)} M_{x,z}\otimes_{\Z} M_{z,y}.\]
The image of $(x,y)$ under the functor $M\otimes_{\Cc} N$ is then defined as the quotient of $P_{x,y}$ by all relations of the form $ad\otimes b=a\otimes db$ where $d\in \Dd(z,z'), a\in M_{x,z},b\in M_{z',y}$. 

A morphism $c\in \Cc(x,x')$ acts by $f.(a\otimes b):=(fa)\otimes b$. A morphism $e\in \Ee(y,y')$ acts by $(a\otimes b).e:=a\otimes(be)$.

Like in the usual case, we recover an adjunction between the tensor product and the $\ho$ functor. This adjointness property directly gives right-exactness of the tensor product.

\begin{prop}\label{tenshom} Let $M,N,Q$ be a $\Cc,\Dd$-bimodule, a $\Dd,\Ee$-bimodule and a $\Cc,\Ee$-bimodule, respectively. We have
\[\ho_{\Cc,\Ee}(M\otimes_\Cc N,Q)\simeq \ho_{\Dd,\Ee}(N,\ho_\Cc(M,Q))\]
where $\ho_{\Cc}(M,Q)$ denotes the $\Dd,\Ee$ bimodule sending $z,y$ to $\ho_\Cc(M_{.,z},Q_{.,y})$.\newline 
This isomorphism is natural and induces an adjunction.
\end{prop}
\begin{proof}
Let $\eta:M\otimes_\Cc\ N\to Q$ be a natural transformation of functors. Let also $(x,y)\in \Ob(\Cc\times \Ee^{op})$, $c\in \Cc(x',x)$ and $e\in \Ee(y,y')$. We have a commutative square
\[\xymatrixrowsep{3pc}\xymatrixcolsep{3pc}\xymatrix{(M\otimes_\Cc N)_{x,y} \ar[r]^-{\eta_{x,y}} \ar[d]_-{(M\otimes_\Cc N)(c,e)}& Q_{x,y} \ar[d]^-{Q(c,e)} \\ (M\otimes_\Cc N)_{x',y'} \ar[r]_-{\eta_{x',y'}} & Q_{x',y'}}\]
which we summarize in the following formula : $\eta(ca\otimes be)=c.\eta(a\otimes b).e$ for all $a\in M$ and $b\in N$.

Let now $b\in N_{z,y}$ for some $z\in \Ob(\Dd)$ and some $y\in \Ob(\Ee)$. For $x\in \Ob(\Cc)$, we have a morphism
\[\begin{array}{cccc} \varphi(b)_x:&M_{x,z} & \longrightarrow &Q_{x,y} \\& a&\longmapsto &\eta (a\otimes b)\end{array}\]
which induces a natural transformation $\varphi(b):M_{.,z}\Rightarrow Q_{.,y}$.

Conversely, for $\psi\in \ho_{\Dd,\Ee}(N,\ho_\Cc(M,Q))$, we have a natural transformation $\varepsilon:M\otimes_\Cc N\Rightarrow Q$ given by
\[\begin{array}{cccc} \varepsilon_{x,y}:&(M\otimes N)_{x,y}&\longrightarrow& Q_{x,y}\\ & a\otimes b&\longmapsto& \psi(b)_{x}(a)\end{array}\]
which gives the inverse bijection.
\end{proof}

\begin{cor}
Let $N$ be a $\Dd,\Ee$ bimodule, then the functor $-\otimes_\Dd N:\Cc -\modu- \Dd\to \Cc- \modu- \Ee$ is a right-exact functor.
\end{cor}

In particular we can define $\Tor_n^{\Z\Dd}(-,N)$ to be the $n$-th left-derived functor of the tensor product $-\otimes_\Dd N$.

\begin{definition}
Let $\Cc,\Dd$ be categories, and let $M$ be a $\Cc,\Dd$-bimodule. The $n$-th \nit{homology group} of $\Cc$ with coefficients in $M$ is defined as
\[H_n(\Cc,M):=\Tor_{n}^{\Z\Cc}(\Z,M)\]
It is endowed with a structure of (right)-$\Dd$-module.
\end{definition}

In the case where $\Cc=G$ is a group, and $A$ is a $\Z G$-module, we recover the classical definition of the homology of $G$ with coefficients in $A$. In this case, one can check that $\Z\otimes_{\Z G}-$ is isomorphic to the \emph{functor of coinvariants}, sending a $\Z G$-module $A$ to 
\[A_G:=\bigslant{A}{\langle g.a-a~|~g\in G,a\in A\rangle}\]
(see for instance \cite[Section 6.1]{weibel}).

Back to the general case, if we specialize the above definition to the case where $\Dd=\Z$, then the tensor product $\Z\otimes_{\Z \Cc} M$ is constructed by considering the quotient of the direct sum
\[P:=\bigoplus_{z\in \Ob(\Cc)} M_z\]
by all relations of the form $f.a=a$ for $a\in M_z$ and $f\in \Cc(-,z)$.  We see that this construction is reminiscent of the functor of coinvariants in group homology.

\begin{rem}
The arguments of \cite[Section 6.5]{weibel} can be adapted to give a notion of ``bar resolution'' for categories. The bar resolution induces in turn a canonical resolution of the trivial $\Z\Cc$-module.
\end{rem}

\subsection{Category, groupoid and group}\label{cgg}

Just like for monoids, one can construct the \nit{enveloping groupoid} $\Gg(\Cc)$ of a category $\Cc$ by formally inverting all the morphisms of $\Cc$ (see \cite[Section II.3.2]{ddgkm}). Depending on the properties of $\Cc$, the enveloping groupoid $\Gg(\Cc)$ can be more or less hard to describe. In particular there is no reason that the homology of $\Cc$ should be the same as that of $\Gg(\Cc)$. 

\begin{exemple}(\cite[Theorem 1]{mcduff}) Every path-connected space has the same weak homotopy type as the classifying space of some discrete monoid. This is obviously false if we replace ``monoid'' by ``group'', since the classifying space of a discrete group is always a $K(\pi,1)$ space.
\end{exemple}

Let $\Cc$ be a category, and let $\Gg:=\Gg(\Cc)$ be its enveloping groupoid. As $\Gg$ and $\Cc$ have the same set of objects, one can consider $\Z\Gg$ as a $\Gg,\Cc$ bimodule, sending a pair of objects $(x,y)$ to $\Z\Gg(x,y)$.

We have a ``scalar restriction'' functor $\Z\Gg-\modu\to \Z\Cc-\modu$ coming from the canonical functor $\Cc\to \Gg$. We also have an ``inversion of scalars'' functor $\Z\Cc-\modu\to \Z\Gg-\modu$ given by $\Z\Gg\otimes_\Cc -$.

\begin{lem}
For every category $\Cc$, the scalar inversion functor $\Z\Cc-\modu\to \Z\Gg-\modu$ is left-adjoint to the scalar restriction functor $\Z\Gg-\modu\to \Z\Cc$.
\end{lem}
\begin{proof}
This comes from the tensor-hom adjunction of Proposition \ref{tenshom}. Let $M$ and $Q$ respectively be a $\Cc$-module and a $\Gg$-module. We have
\[\ho_\Gg(\Z\Gg\otimes_\Cc M,Q)\simeq \ho_{\Cc}(M,\ho_\Gg(\Z\Gg,Q)).\]
The $\Cc$-module $\ho_{\Gg}(\Z\Gg,Q)$ is isomorphic to the scalar restriction functor (by the Yoneda Lemma).
\end{proof}

This lemma is somewhat reminiscent of the Frobenius reciprocity. Except here, instead of inducing from a subgroup to an ambient group, we start from a category and we induce to its enveloping groupoid. This adjointness property implies in particular that the functor $\Z_\Gg\otimes_\Cc -$ is right-exact. 
 
We want a condition under which the scalar inversion functor is not only right-exact (which always holds), but also left-exact. This would ensure that this functor preserves homology. A good such condition is given by the notion of an Ore category. 

\begin{definition}(\cite[Definition II.1.14 and Definition II.3.10]{ddgkm}) \label{leftore}\newline We say that a category $\Cc$ is
\begin{enumerate}[-]
\item \nit{left-cancellative}(resp. \nit{right-cancellative}) if every relation $fg=fg'$ (resp. $gf=g'f$) with $f,g,g'\in \Cc$ implies $g=g'$. The category $\Cc$ is \nit{cancellative} if it is both left- and right-cancellative.
\item a \nit{left-Ore category} if it is cancellative and any two elements with the same target admit a common left-multiple.
\end{enumerate}
\end{definition}

This classical notion allows for a convenient description of the enveloping groupoid $\Gg(\Cc)$ in terms of left-fractions.

\begin{lem}(\cite[Proposition II.3.11]{ddgkm})\newline 
If $\Cc$ is a left-Ore category, then the enveloping groupoid $\Gg$ of $\Cc$ can be described as a groupoid of fractions: the functor $\Cc\to \Gg$ is injective on morphisms and every morphism in $\Gg$ has the form $f^{-1}g$ for some $f,g\in \Cc$.
\end{lem}

This convenient description of the morphisms $\Gg(\Cc)$ has a remarkable consequence on modules of the form $\Z \Gg(x,-)\otimes_\Cc -$ for $x\in \Ob(\Cc)$. This in turn induces the exactness of the scalar inversion functor $\Z \Gg\otimes_\Cc -$.

\begin{prop}Let $x\in \Ob(\Cc)$. If $\Cc$ is a left-Ore category, then the right $\Cc$-module $\Z\Gg(x,-)\otimes_\Cc-$ is a direct limit of free $\Cc$-modules. In particular it is a flat module.
\end{prop}
\begin{proof}
Our argument is a categorical rephrasing of the proof of \cite[Theorem 2.3]{squier}.

For every morphism $f\in \Cc(y,x)$, precomposition by $f^{-1}$ induces a morphism of right $\Cc$-module $\varphi_f:\Z\Cc(y,-)\to \Z\Gg(x,-)$. Since $\Cc$ is a left-Ore category, every morphism in $\Gg(x,-)$ can be described as a fraction. This means that $\Gg(x,-)$ is the union of the images of the morphisms $\varphi_f$ for $f\in \Cc(-,x)$. Furthermore, the system given by the $\varphi_f$ is a directed system because of the existence of left-multiples.

Just like in the usual case, the functor $\Tor$ commutes with inductive limits, which gives the flatness of $\Z\Gg(x,-)$.
\end{proof}

\begin{theo}\label{2.16}
Let $\Cc$ be a left-Ore category with enveloping groupoid $\Gg$. The $\Gg,\Cc$-bimodule $\Z\Gg$ is a flat $\Cc$-module.
\end{theo}
\begin{proof}
We already know that, as a left-adjoint, the functor $\Z\Gg\otimes_{\Cc}-$ is right-exact. We only have to show that it is left-exact, that is it preserves kernels. \newline Let $M$ be a $\Cc$-module, the abelian group $\Z\Gg(x,-)\otimes_\Cc M$ is a quotient of the direct sum
\[\bigoplus_{y\in \Cc}\Z\Gg(x,y)\otimes_\Z M_y\]
We see that the abelian group $\Z\Gg(x,-)\otimes_\Cc M$ is the image of $x$ under the functor $\Z\Gg\otimes_\Cc M$. As kernels in a functor category are computed objectwise, the flatness of $\Z\Gg(x,-)$ induces the flatness of $\Z\Gg$ as claimed.
\end{proof}

This is the exactness property we were looking for. A first consequence is that, under the assumption that $\Cc$ is left-Ore, the scalar inversion functor preserves homology.

\begin{cor}\label{cor2.16}
Let $\Cc$ be a left-Ore category. For every $\Cc$-module $M$ and every $n\in \Z_{\geqslant 0}$ we have $H_n(\Cc,M)=H_n(\Gg(\Cc),\Z\Gg\otimes_\Cc M)$.
\end{cor}
\begin{proof}Let $M$ be a $\Cc$-module. By definition we have $H_n(\Cc,M)=\Tor_n^{\Z\Cc}(\Z,M)$. By Theorem \ref{2.16}, this is equal to $\Tor_n^{\Z\Gg}(\Z,\Z\Gg\otimes_{\Z\Cc}M)=H_n(\Z\Gg,M)$.\end{proof}

Now that we have equality between the homology of a category and that of its enveloping groupoid (under suitable assumptions), we want equality between the homology of a groupoid and that of a group to which it is equivalent as a category.

Let $\Gg$ be a groupoid, that we assume to be connected from now on. We also fix an object $x_0\in \Ob(\Gg)$ and set $G:=\Gg(x_0,x_0)$. 

We denote by $\iota$ the inclusion functor $G\to \Gg$. The choice, for every $x\in \Ob(\Gg)$, of a morphism $u_x\in \Gg(x_0,x)$ induces a functor $\pi:\Gg\to G$, sending $f:x\to y$ to $u_{x}fu_{y}^{-1} \in G$. The functors $\iota$ and $\pi$ are quasi-inverse equivalences of categories. Indeed $\pi\circ \iota$ is the identy morphism of $G$, and $\iota\circ \pi$ is fully faithful, and essentially surjective since $\Gg$ is a connected groupoid. The equivalences $\iota$ and $\pi$ induce in turn equivalences between the categories of $\Gg$-modules and of $G$-modules.

In practice, if $M$ is a $G$-module, then the induced $\Gg$-module sends every object to $M$ and a morphism $f$ acts by $u_xf u_y^{-1}$. We also denote this module by $M$.
\begin{prop}\label{groupoidgroup}
Let $\Gg$ be a connected groupoid, equivalent to a group $G$. For every $G$-module $M$ and every $n\in \Z_{\geqslant 0}$ we have $H_n(G,M)=H_n(\Gg,M)$.
\end{prop}
\begin{proof}
The functor $\pi:\Gg\to G$ induces an equivalence of categories $\pi_*:\Z G-\modu\to \Z\Gg$, which is in particular an exact functor. For a $G$-module $M$, we get
\begin{align*}
H_n(G,M)&=\Tor_n^{\Z G}(\Z,M)\\
&=\Tor_n^{\Z\Gg}(\pi_*(\Z),\pi_*(M))\\
&=\Tor_n^{\Z\Gg}(\Z,M)=H_n(\Gg,M)
\end{align*}
as claimed.
\end{proof}

\section{The Dehornoy-Lafont order complex for categories}\label{sec3}
Garside categories were originally introduced as a natural generalization of Garside mon-\newline oids (see for instance \cite{kgarcat}). A comprehensive survey of the general theory of Garside categories is done in \cite{ddgkm}.

One of the (many) uses of a Garside structure on a category is that it gives rise to convenient resolutions allowing for the computation of the homology of the category (see \cite[Section III.3.4]{ddgkm}). These resolutions, are generalizations of previous works on Garside monoids (see \cite{dehlaf} and \cite{cmw}).

However, the order complex of Dehornoy and Lafont, which we will call the Dehornoy-Lafont complex from now on, has not been adapted to the case of a category in \cite{ddgkm}. This complex (in the case of a monoid) is smaller in size, so it is usually more suitable for practical purposes. We propose in this section a generalization of this complex to a categorical setting.

\subsection{The complex}\label{sec:2.1}
In the case of a monoid, the Dehornoy-Lafont complex was introduced in \cite[Section 4]{dehlaf} under the name ``order resolution''. It is based on a well-ordering of some generating system of the monoid. Although we plan to apply this new complex to a Garside category, the definition can be formulated in a slightly more general case mimicking the definition of Gaussian monoid.

We start with a small category $\Cc$. We define a relation $\preceq$ on $\Cc(x,-)$ by
\[\forall f,g\in \Cc(x,-),~ f\preceq g\Leftrightarrow \exists h~|~ fh=g.\]
In particular, the source of $h$ must be the target of $f$, and its target must be the target of $g$. We say that $g$ is a right-multiple of $f$ and that $f$ left-divides $g$. Likewise, one can define a relation $\succeq$ on $\Cc(-,x)$.

The relation $\preceq$ is obviously reflexive and transitive. But it is not necessarily antisymmetric: if $x$ is an invertible morphism, then we have $x\prec 1 \prec x$. Because of this we always require $\Cc$ to have no nontrivial invertible morphism from now on.

\begin{definition}(\cite[Definition II.2.26]{ddgkm}) \newline Let $\Cc$ be a category with no nontrivial invertible morphisms. The category $\Cc$ is said to be \nit{right-Noetherian} (resp. \nit{left-Noetherian}) if the relation $\preceq$ (resp. $\succeq$) admits no infinite strictly descending chain. The category $\Cc$ is \nit{Noetherian} if it is both left-and right-Noetherian.
\end{definition}

\begin{rem}
The definition stated in \cite{ddgkm} is slightly broader than the one we state here, as the authors need to consider the case where $\Cc$ admits nontrivial isomorphisms. In the case studied here where $\Cc^\times$ is reduced to the identities, the two definitions coincide.
\end{rem}

If $\Cc$ is left-cancellative and left-Noetherian (and $\Cc^\times$ is reduced to the identity morphisms), then one can see that $\preceq$ is a partial order. Likewise, $\succeq$ is a partial order if $\Cc$ is right-cancellative and right-Noetherian (and $\Cc^\times$ is reduced to the identity morphisms).

Following \cite[Definition II.2.52]{ddgkm}, a morphism $a\in \Cc$ is called an \nit{atom} if its only left-divisors are itself and the identity. If $\Cc$ is Noetherian, then every morphism in $\Cc$ is a composition of a finite number of atoms. In this case we also get that a subfamily of $\Cc$ generates $\Cc$ if and only if it contains all of the atoms of $\Cc$.

\begin{definition}(\cite[Definition II.2.9 and Definition II.2.20]{ddgkm})\newline 
Let $\Cc$ be a category, and let $f,g$ be two morphisms with the same target. A \nit{left-lcm} of $f$ and $g$ is a common left-multiple of $f$ and $g$ which right-divides any left-multiple of $f$ and $g$.\newline
The category $\Cc$ admits \nit{conditional left-lcms} if any two elements of $\Cc$ that admits a common left-multiple admit a left-lcm. \newline Following \cite[Section 1.1]{dehlaf}, a right-cancellative right-Noetherian category which admits conditional left-lcms is called \nit{locally left-Gaussian}. If $\Cc$ furthermore admits left-lcms, $\Cc$ is called \nit{left-Gaussian}.
\end{definition}

\begin{rem}\label{gargauss} A Garside category, that is a cancellative Noetherian category endowed with a Garside map (see \cite[Definition V.2.19]{ddgkm}), is in particular a Gaussian category. This is why we can apply the Dehornoy-Lafont complex to Garside categories.
\end{rem}

From now on, we consider a locally left-Gaussian category $\Cc$. The proofs and construction are direct adaptations of the arguments of \cite[Section 4]{dehlaf} to the case of a category. 

We start by fixing a finite set of morphisms $\Aa$, which generates $\Cc$. We also fix, for every object $x$ of $\Cc$, a linear ordering $<$ on the set $\Aa(-,x)$ of elements of $\Aa$ with target $x$. In practice this amounts to fixing a linear ordering on the set $\Aa$, that we then restrict to the sets $\Aa(-,x)$.

For every morphism $f\in \Cc$, the set of elements of $\Aa$ dividing $f$ on the right is ordered by $<$. Since $\Aa$ is finite, one can define $\md(f)$ to be the $<$-least right-divisor of $f$ in $\Aa$.

\begin{definition}
Let $n$ be an integer. A $n$-\nit{cell} is a $n$-tuple $[\alpha_1,\ldots,\alpha_n]$ of elements of $\Aa$ sharing the same target such that $\alpha_1<\ldots<\alpha_n$ and
\[\forall i \in \intv{1,n}, \alpha_i=\md(\lcm(\alpha_i,\ldots,\alpha_n))\]
We say that a $n$-cell $[\alpha_1,\ldots,\alpha_n]$ has source $x\in \Ob(\Cc)$ if $x$ is the source of $\lcm(\alpha_1,\ldots,\alpha_n)$. For $x\in \Ob(\Cc)$, we define $(\Xx_n)_x$ to be the set of $n$-cells with source $x$.

We define $C_n$ to be the free $\Z\Cc$-module associated to the family $\{(\Xx_n)_x\}_{x\in \Ob(\Cc)}$
\end{definition}

In particular we see that, for each object $x$, we have $(\Xx_0)_x=\{[\varnothing]\}$ and $(\Xx_1)_x=\Aa(x,-)$. To avoid confusion, we will alternatively denote by $[\varnothing]_x$ the only element of $(\Xx_0)_x$.

Let $x$ be an object of $\Cc$. By definition of a free module over a category (see Section \ref{modcat}), $(C_n)_x$ is generated as an abelian group by elements of the form $f[A]$, where $f\in \Cc(x,y)$ and $[A]$ is a $n$-cell with source $y$. We call such elements \nit{elementary $n$-chains}.

Like in the case of a monoid, the following preodering on elementary $n$-chains will allow us to use induction arguments.

\begin{definition} We denote by $A_1$ the first element of a nonempty tuple $[A]$. Let $f[A]$ and $g[B]$ be elementary $n$-chains with same source. We say that $f[A] \sqsubset g[B]$ holds if we have either $f\lcm(A)\prec g\lcm(B)$ or $n>0$, $f\lcm(A)=g\lcm(B)$ and $A_1<B_1$.\newline If $\sum f_i[A_i]$ is an arbitrary $n$-chain, we say that $\sum f_i[A_i]\sqsubset g[B]$ holds if $f_i[A_i]\sqsubset g[B]$ holds for every $i$.
\end{definition}

\begin{lem}\label{4.2}
For every $n$, the relation $\sqsubset$ on $n$-dimensional elementary chains with same source is compatible with multiplication on the left, and it has no infinite decreasing sequence.
\end{lem}
\begin{proof}
Assume $f[A]\sqsubset g[B]$, and let $h$ be a morphism in $\Cc$. Then $f\lcm(A)\prec g\lcm(B)$ implies $hf[A]\prec hg[B]$, and $f\lcm(A)=g\lcm(B)$ implies $hf\lcm(A)=hg\lcm(B)$. We have $hf[A]\sqsubset hg[B]$ in each case.

Let now $\ldots\sqsubset f_2[A_2]\sqsubset f_1[A_1]$ be a descending sequence. We have a descending sequence $\ldots \preceq f_2\lcm(A_2)\preceq f_1\lcm(A_1)$. The left-Noetherianity of $\Cc$ gives that this sequence is stationary: there is some $i_0$ such that we have $f_i\lcm(A_i)=f_{i+1}\lcm(A_{i+1})$ for $i\geqslant i_0$.
By definition of $\sqsubset$, we must have $(A_{i+1})_1<(A_i)_1$ for $i\geqslant i_0$. The sequence $(A_{j})_1$ for $j\geqslant i_0$ is $<$-decreasing. But $(A_i)_1$ right-divides $\lcm(A_i)$ and $f_i\lcm(A_i)$. So the sequence $(A_{j})_1$ for $j\geqslant i_0$ is a $<$-decreasing sequence of divisors of $f_{i_0}\lcm(A_{i_0})$. The sequence $(A_{j})_1$ thus is stationary, and the sequence $f_i[A_i]$ is stationary.
\end{proof}

We will now define the differential map $\partial_n:C_n\to C_{n-1}$ along with a contracting homotopy $s_n:C_n\to C_{n+1}$ and a so-called reduction map $r_n:C_n\to C_n$. The map $\partial_n$ is $\Z\Cc$-linear, whereas $s_n$ and $r_n$ are only $\Z$-linear.

\begin{definition}
Let $f[A]$ be an elementary chain. We say that $f[A]$ is \nit{irreducible} if either $f[A]=1_x[\varnothing]_x$ or $A_1=\md(f\lcm(A))$. Otherwise we say that $f[A]$ is \nit{reducible}.
\end{definition}

The construction of $\partial_*,r_*$ and $s_*$ uses induction on $n$. The induction hypothesis, denoted $(H_n)$ is the conjunction of the following two statements, where $r_n:=s_{n-1}\circ \partial_n$
\[\begin{array}{lc} (P_n)& \partial_n(r_n(f[A]))=\partial_n(f[A]) \\ (Q_n)& r_n(f[A])\begin{cases} =f[A] & \text{if } f[A] \text{ is irreducible}\\ \sqsubset f[A] & \text{if }f[A]\text{ is reducible}\end{cases} \end{array}\]

In degree $0$, the construction is usual and straightforward: we define $\partial_0:C_0\to \Z$ and $s_{-1}:\Z\to C_0$ by
\[\forall x\in \Ob(\Cc),~~\partial_0([\varnothing]_x):=1\text{ and } s_{-1}(1)=[\varnothing]_x\]

\begin{lem}
Property $(H_0)$ is satisfied.
\end{lem}
\begin{proof}
The mapping $r_0:=s_{-1}\circ \partial_0$ is $\Z$-linear with
\[r_0(f[\varnothing]_x)=s_1(\partial_0(f[\varnothing]_x))=[\varnothing]_y\]
for $f \in \Cc(y,x)$. Hence we obtain
\[\partial_0(r_0(f[\varnothing]_x))=\partial_0([\varnothing]_y)=1,~~\partial(f[\varnothing]_x)=f.1=1\]
because of the structure of the trivial $\Cc$-module $\Z$. Thus $(P_0)$ holds. For $(Q_0)$, we know that an elementary $0$-chain $f[\varnothing]_x$ is irreducible if and only if $f=1_x$, in which case we have $r_0(f[\varnothing]_x)=r_0([\varnothing]_x)=[\varnothing]_x$. Otherwise, we have $r_0(f[\varnothing]_x)=[\varnothing]_y\sqsubset f[\varnothing]_x$ by definition. Thus $(Q_0)$ also holds.
\end{proof}

We now assume that $(H_n)$ is satisfied, in particular we assume that both $\partial_n$ and $r_n$ have been constructed. We must now define
\[\partial_{n+1}:C_{n+1}\to C_n,~~s_n:C_{n}\to C_{n+1}, ~~r_{n+1}=s_n\circ \partial_{n+1}:C_{n+1}\to C_{n+1}\]
and show that $(H_{n+1})$ is satisfied. In the sequel, we use the notation $[\alpha,A]$ to write $(n+1)$-cells. By definition, saying that $[\alpha,A]$ is a $(n+1)$-cell amounts to saying that $A$ is a $n$-cell and $\alpha=\md(\lcm(\alpha,A))$. We denote by $\alpha_{/A}$ the morphism defined by the equation $(\alpha_{/A})\lcm(A)=\lcm(\alpha,A)$.

\begin{definition}\label{dl}
\begin{enumerate}[$\bullet$]
\item We define the morphism of $\Z\Cc$-module $\partial_{n+1}:C_{n+1}\to C_n$ by 
\[\partial_{n+1}([\alpha,A]):=\alpha_{/A}[A]-r_n(\alpha_{/A}[A])\]
\item We inductively define the $\Z$-linear map $s_n:C_n\to C_{n+1}$ by
\[s_n(f[A]):=\begin{cases} 0 & \text{if } f[A]\text{ is irreducible}\\ g[\alpha,A]+s_n(gr_n(\alpha_{/A}[A])) & \text{otherwise, with }\alpha=\md(f\lcm(A))\text{ and }f=g\alpha_{/A}\end{cases}\]
\item Finally, we define $r_{n+1}:C_{n+1}\to C_{n+1}$ by $r_{n+1}=s_n\circ \partial_{n+1}$.
\end{enumerate}
\end{definition}

We can first notice that, under these definitions, $\partial_{n+1}$, $s_n$ and $r_n$ all preserve the source of $n$-cells.

The definition of $\partial_{n+1}$ is direct (since $r_n$ has been constructed in order to satisfy property $(H_n)$). The definition of $s_n$ is inductive and we must check that it is well founded. Let $f[A]$ be a reducible chain. The chain $\alpha_{/A}[A]$ appearing in the definition of $s_n$ is also reducible since $\alpha<A_1$ holds by definition. Thus $(Q_n)$ gives $r_n(\alpha_{/A}[A])\sqsubset \alpha_{/A}[A]$ and
\[gr_n(\alpha_{/A}[A])\sqsubset g\alpha_{/A}[A]=f[A].\]
Since the relation $\sqsubset$ admits no infinite decreasing sequence, the definition of $s_n$ (and of $r_{n+1}$) is then well founded.

We can also check that $\partial_*$ induces a chain complex. Let $[\alpha,A]$ be a $(n+1)$-cell, we have
\[\partial_n\partial_{n+1}[\alpha,A]=\partial_n(\alpha_{/A}[A])-\partial_n(r_n(\alpha_{/A}[A]))=0\]
because of $(P_n)$. 

The following lemma is useful for showing that $(H_n)$ implies $(P_{n+1})$, but it also contains the information that $(s_n)$ provides a contracting homotopy for the Dehornoy-Lafont complex (see the proof of Proposition \ref{2.15}).

\begin{lem}\label{4.4} Let $f[A]$ be an elementary $n$-chain. Assuming $(H_n)$ we have
\[\partial_{n+1}s_n(f[A])=f[A]-r_n(f[A])\]
\end{lem}
\begin{proof}
We use a $\sqsubset$-induction on $f[A]$. If $f[A]$ is irreducible, then by $(Q_n)$ we have
\[\partial_{n+1}s_n(f[A])=0=f[A]-r_n(f[A]).\]
Assume now that $f[A]$ is reducible. With the notation of Definition \ref{dl}, we obtain
\[\partial_{n+1}s_n(f[A])=g\partial_{n+1}([\alpha,A])-\partial_{n+1}s_n(gr_n(\alpha_{/A}[A])).\]
By $(Q_n)$, we have $gr_n(\alpha_{/A}[A])\sqsubset f[A]$, so the induction hypothesis gives us
\[\partial_{n+1}s_n(gr_n(\alpha_{/A}[A]))=gr_n(\alpha_{/A}[A])-r_n(gr_n(\alpha_{/A}[A])).\]
Applying $(P_n)$ we deduce
\begin{align*}
r_n(gr_n(\alpha_{/A}[A]))&=s_{n-1}(g\partial_n(r_n(\alpha_{/A}[A])))\\
&=s_{n-1}(g\partial_n(\alpha_{/A}[A]))\\
&=r_n(g\alpha_{/A}[A])=r_n(f[A]).
\end{align*}
And so
\[\partial_{n+1}s_n(f[A])=g\alpha_{/A}[A]-gr_n(\alpha_{/A}[A]+gr_n(\alpha_{/A}[A])-r_n(f[A])=f[A]-r_n(f[A])\]
as expected.
\end{proof}

\begin{lem}\label{lem:H_n->P_n+1}
Assuming $(H_n)$, $(P_{n+1})$ is satisfied.
\end{lem}
\begin{proof}
Let $[\alpha,A]$ be an elementary $(n+1)$-chain. We find
\begin{align*}
\partial_{n+1}(r_{n+1}(f[A]))&=\partial_{n+1}s_n\partial_{n+1}(f[A])\\
&=\partial_{n+1}(f[A])-r_n(\partial_{n+1}(f[A]))\\
&=\partial_{n+1}(f[A])-s_{n-1}\partial_n\partial_{n+1}(f[A])\\
&=\partial_{n+1}(f[A])
\end{align*}
\end{proof}

Before we show that $(H_n)$ implies $(H_{n+1})$, we need the following lemma, which substantiates the behavior of $s_n$ relative to $\sqsubset$.

\begin{lem}\label{poti}
Let $f[\alpha,A]$ be a reducible $(n+1)$-chain. For each reducible $n$-chain $g[B]$ satisfying $g\lcm(B)\preceq f\lcm(\alpha,A)$, we have $s_n(g[B])\sqsubset f[\alpha,A]$.
\end{lem}
\begin{proof}
We use $\sqsubset$-induction on $g[B]$. By definition we have
\[s_n(g[B])=h[\beta,B]+s_n\left(\sum h_i[C_i]\right)\]
with $\beta=\md(g\lcm(B))$, $g\lcm(B)=h\lcm(\beta,B)$ and $\sum h_i[C_i]=hr_n(\beta_{/B}[B])$. Furthermore, since we know that $hr_n(\beta_{/B}[B])\sqsubset g[B]$, we always have $h_i[C_i]\sqsubset g[B]$, hence in particular, $h_i\lcm(C_i)\preceq g\lcm(B)\preceq f\lcm(\alpha,A)$. The induction hypothesis gives $s_n(h_i[C_i])\sqsubset f[\alpha,A]$ if $h_i[C_i]$ is reducible. If $h_i[C_i]$ is irreducible, then the contribution of $s_n(h_i[C_i])=0$ to the sum defining $s_n(hr_n(\beta_{/B}[B])$ is trivial. In both cases, it only remains to compare $h[\beta,B]$ and $f[\alpha,A]$.

Two cases are possible. Assume first $g\lcm(B)\prec f\lcm(\alpha,A)$. By construction, we have $h\lcm(\beta,B)=g\lcm(B)$, so we deduce $h\lcm(\beta,B)\prec f\lcm(\alpha,A)$ and therefore $h[\beta,B]\sqsubset f[\alpha,A]$.

Assume now $g\lcm(B)=f\lcm(\alpha,A)$. By construction, $\beta$ is the $<$-least right-divisor of $g\lcm(B)$, hence of $f\lcm(\alpha,A)$. The hypothesis that $f[\alpha,A]$ is reducible means that $\alpha$ is a right-divisor of the latter element, but is not its least right-divisor, so we must have $\beta<\alpha$. This gives $h[\beta,B]\sqsubset f[\alpha,A]$ by definition. 
\end{proof}

\begin{lem}
Assuming $(H_n)$, $(H_{n+1})$ is satisfied.
\end{lem}
\begin{proof}
By Lemma \ref{lem:H_n->P_n+1}, it only remains to prove $(Q_{n+1})$. Let $f[\alpha,A]$ be an elementary $(n+1)$-chain. By definition, we have
\[r_{n+1}(f[\alpha,A])=s_n(f\alpha_{/A}[A])-s_n\left(\sum g_i[B_i]\right)\]
with $\sum g_i[B_i]=fr_n(\alpha_{/A}[A])$. If $f[\alpha,A]$ is irreducible, then we have $\alpha=\md(x\lcm(\alpha,A))$. The definition of $s_n$ gives
\[s_n(f\alpha_{/A}[A])=f[\alpha,A]+s_n\left(\sum g_i[B_i]\right)\]
and we deduce $r_{n+1}(f[\alpha,A])=f[\alpha,A]$.

Assume now that $f[\alpha,A]$ is reducible. First, we have $f\alpha_{/A}\lcm(A)=f\lcm(\alpha,A)$, so Lemma \ref{poti} gives $s_n(f\alpha_{/A}[A])\sqsubset f[\alpha,A]$. Since $\alpha=\md(\lcm(\alpha_{/A},A)<A_1$, the chain $\alpha_{/A}[A]$ is reducible, so property $(Q_n)$ gives $r_n(\alpha_{/A}[A])\sqsubset \alpha_{/A}[A]$. Hence Lemma \ref{4.2} gives $xr_n(\alpha_{/A}[A])\sqsubset f\alpha_{/A}[A]$, i.e, $g_i[B_i]\sqsubset f\alpha_{/A}[A]$ for all $i$. This implies in particular $g_i\lcm(B_i)\preceq f\alpha_{/A}\lcm(A)=f\lcm(\alpha,A)$. Applying Lemma \ref{poti} to $g_i[B_i]$ gives $s_n(g_i[B_i])\sqsubset f[\alpha,A]$. We deduce that $r_{n+1}(f[\alpha,A])\sqsubset f[\alpha,A]$, which is property $(Q_{n+1})$.
\end{proof}

Thus, the induction hypothesis is maintained, and the construction can be carried out. We can now state the main result of this section: the Dehornoy-Lafont complex for a Gaussian category provides a free resolution of the trivial module.

\begin{prop}\label{2.15}
Let $\Cc$ be a locally left-Gaussian category. The complex $(C_*,\partial_*)$ is a free resolution of the trivial $\Z\Cc$-module $\Z$.
\end{prop}
\begin{proof}
We already have seen that $(C_*,\partial_*)$ is a complex of $\Z\Cc$-modules. The formula of Lemma \ref{4.4} rewrites into
\[\partial_{n+1}s_n+s_{n+1}\partial_n=1\]
which shows that $s_*$ is a contracting homotopy.
\end{proof}

Combining this proposition and Theorem \ref{2.16}, we get an analogue for groupoids. This is the result we will use in Section \ref{sec:3}.

\begin{prop}
Let $\Cc$ be a cancellative left-Gaussian category. The complex \newline $(\Z\Gg(\Cc)\otimes_{\Z\Cc}C_*,\partial_*)$ is a free resolution of the trivial $\Z\Gg$-module $\Z$.
\end{prop}

\subsection{Reduction of computations}\label{mini}
Let $\Cc$ be a locally left-Gaussian category, and let $\Aa$ be a finite set of morphisms of $\Cc$ which generates $\Cc$. The Dehornoy-Lafont complex depends on the one hand on the structure of $\Cc$ as a category, and on the other hand on the linear order chosen on $\Aa$. Here we propose a computationally efficient solution for constructing an order on $\Aa$ yielding few $2$-cells. This order is not optimal a priori (even for minimizing the number of 2-cells), but it gives good results in practice.

Let $x$ be an object of $\Cc$. We first consider the set $L_x$ of all elements of $\Cc(-,x)$ which are the lcm of a pair of distinct elements of $\Aa$. Our strategy is, for each $\ell\in L_x$, to try to reduce the number of two cells $[a,b]$ with $a\vee b=\ell$.

Let $\ell$ be in $L_x$. One can consider $\Aa_\ell$ the set of elements of $\Aa$ which right-divide $\ell$. This set is included in $\subset\Aa(-,x)$ by definition. For $a\in \Aa_\ell$, we set $n(a,\ell)$ to be the cardinality of the following set
\[\{b\in \Aa_\ell~|~ a\vee b=\ell\}.\]
If $a$ is the $<$-minimum of $\Aa_\ell$, then there are precisely $n(a,\ell)$ $2$-cells of the form $[a,b]$ with $a\vee b=\ell$. In particular we deduce the following lemma.

\begin{lem}\label{4.1}
Let $x$ be an object of $\Cc$. A lower (resp. upper) bound for the number of $2$-cells made of elements of $\Aa$ with target $x$ is given by
\[\sum_{\ell\in L_x} \min_{a\in \Aa_\ell} n(a,\ell)~\left(resp.~\sum_{\ell\in L_x} \max_{a\in \Aa_\ell} n(a,\ell)\right)\]
\end{lem}
In practice, these bounds may or may not be reached. 

\begin{definition}
Let $x$ be an object of $\Cc$ and let $\ell\in L_x$. For $a\in \Aa_\ell$, the \nit{condition} on $\Aa(-,x)$ associated to $a$ and $\ell$ is the set-theoretic relation
\[\{(a,b)~|~b\in \Aa_\ell\}\subset \Aa(-,x)\times \Aa(-,x)\]
We say that such a condition is \nit{optimal} if we furthermore have
\[n(a,\ell)=\min_{b\in \Aa_\ell}n(b,\ell).\]
We say that a family of conditions is \nit{compatible} if their union is a subrelation of an order on $\Aa(-,x)$.
\end{definition}
We can first check that the reflexive closure of a condition is always an order. It is also obvious that two different conditions associated to a same element $\ell\in L_x$ are never compatible: it would mean that $\Aa_\ell$ has two distinct minima.

Testing if a family of conditions is compatible only amounts to testing whether or not its reflexive transitive closure is antisymmetric. This is easy to test in practice due to the form of conditions as set-theoretic relations.

In order to get an adequate order, we try and find a maximal family of compatible conditions yelding few $2$-cells. We propose the following procedure (a detailed code is available at \url{https://github.com/ogarnier/dehornoy_lafont_computations.git}) :
\begin{enumerate}[1.]
\item Set $C:=\varnothing$.
\item Compute $Comp(C)$ the set of conditions on $\Aa(-,x)$ which are compatible with $C$.
\item While $Comp(C)\neq \varnothing$ do \begin{enumerate}[-]
\item Choose a condition $(a,\ell)$ in $Comp(C)$ which minimizes the quantity 
\[n(a,\ell)-\min_{b\in \Aa_\ell} n(b,l).\]
\item Add $(a,\ell)$ to $C$.
\end{enumerate}
\end{enumerate}
The result of this procedure is a maximal family of compatible conditions. This family then induces an order which can be refined into a linear order over $\Aa(-,x)$. Since a choice is made at each step of this procedure, it is very hard to check wether or not the resulting order is optimal for minimizing the number of $2$-cells. The condition chosen at the $n$-th step could for instance be incompatible with an optimal condition at the $n+1$-th step, which would force us to fall back on a less good condition. This would give us an order with more $2$-cells that if we reversed the steps $n+1$ and $n$. Nevertheless, we will see in the next section that this procedure gives rather good results in practice.


Another computational issue arising from the Dehornoy-Lafont complex is the computation of the differential. Since this differential is defined recursively and using the auxiliary morphisms $r_n$ and $s_n$, its calculation may lead to a lot of redundancy.

A first solution is to stock the results of $\partial_n$ applied on the cells and then use the $\Z\Cc$-linearity of $\partial_n$. Unfortunately, one cannot do the same for $s_n$ and $r_n$ as they are not $\Z\Cc$-linear, but only $\Z$-linear. One would theoretically have to stock the results of $r_n$ and $s_n$ applied to every elementary chain, and not only to cells. But, in practice, $r_n$ need only be calculated on chains of the form $\alpha_{/A}[A]$, where $[\alpha,A]$ is a cell (see Definition \ref{dl}). So we can also store the results of $r_n$ on chains of this form to avoid redundant computations.

\section{Homology computations for exceptional complex braid groups}\label{sec:3}

We are now going to use the Dehornoy-Lafont complex to compute the homology of exceptional complex braid groups. Recall that a complex reflection group $W$ is a finite subgroup of $\GL_n(\C)$ generated by (pseudo-)reflections, that is finite order automorphisms of $\C^n$ which pointwise fixes some hyperplane (see \cite{lehrertaylor}). We associate to $W$ the complementary $X$ of the arrangement of all reflecting hyperplanes associated to the reflections of $W$. The action of $W$ on $X$ is free and induces a covering map from $X$ to $X/W$. The complex braid group $B(W)$ (resp. the pure complex braid group $P(W)$) is then defined as the fundamental group of $X/W$ (resp. of $X$). The braid group $B(W)$ is generated by so-called braid reflections, and it admits a length morphism $B(W)\to \Z$, defined by $\sigma\mapsto 1$ for every braid reflections $\sigma$ (\cite[Proposition 2.2 and Proposition 2.16]{bmr}).

Complex reflection groups are known to behave in a ``semi-simple'' way: they can be decomposed as products of irreducible groups (meaning that their representations as subgroups of $\GL_n(\C)$ are irreducible). The same goes for complex braid groups, meaning that we only need to consider braid groups of irreducible complex reflection groups.  

Irreducible complex reflection groups were classified in 1954 by Shephard and Todd, with an infinite series $G(de,e,n)$ depending on three integer parameters, and 34 exceptional groups, labelled $G_4,\ldots,G_{37}$. Note, however, that two distinct reflection groups can have isomorphic braid groups, for instance, the braid groups of $G_{7}$ and $G(12,2,2)$ are isomorphic. This will prove useful in avoiding redundant computations (see Section \ref{isodi}).

The braid group of the exceptional group $G_k$ will be denoted by $B_k$, it should not be confused with the classical braid group on $k$ strands, which will not be considered here. We will restrict our attention to exceptional groups and their braid groups. 

As the construction of the Dehornoy-Lafont complex relies on some underlying Gaussian category (or monoid), we will use various Garside structures for exceptional braid groups (see Remark \ref{gargauss}).

Once we have computed the Dehornoy-Lafont complex, we will use it to compute the homology of $B(W)$ with coefficients in the following modules $M$:
\begin{enumerate}[-]
\item $M=\Z$ is the trivial $B(W)$-module.
\item $M=\Z$, where the braid reflections of $B(W)$ act by $-1$. We denote this module by $\Z_\epsi$, we call it the \nit{sign representation} of $B(W)$.
\item $M=k[t,t^{-1}]$, where $k$ is a field and the braid reflections of $B(W)$ act by $t$. We consider the case where $k=\Q$ or $k$ is some finite field.
\end{enumerate}
As pointed out in \cite{calmil}, the homology $H_*(B(W),k[t,t^{-1}])$ where $k$ is a field can be identified with the homology of the Milnor fiber of the singularity corresponding to $W$.

\subsection{Isodiscriminantality}\label{isodi} Let $W$ and $W'$ be two irreducible complex reflection groups. By the Chevalley-Shephard-Todd Theorem (see \cite[Theorem 3.20]{lehrertaylor}), one can choose a family of homogeneous polynomials $f_1,\ldots,f_n$ such that the algebra $\C[X_1,\ldots,X_n]^W$ of $W$-invariant polynomials is a polynomial algebra generated by $f_1,\ldots,f_n$. The sequence $f_1,\ldots,f_n$ is called a system of basic invariants for $W$. The polynomial map $f=(f_1,\ldots,f_n):\C^n\to \C^n$ induces a map $\hhat{f}:\C^n/W\to \C^n$ which sends $X/W$ to the complementary of an algebraic hypersurface $\Hh$. The hypersurface $\Hh$ is the image under $f$ of the union of the reflecting hyperplanes.

Consider $W,W'\leqslant \GL_n(\C)$ two complex reflection groups with two systems of basic invariants $f_1,\ldots,f_n$ and $f_1',\ldots,f_n'$ for which the associated discriminant hypersurfaces are the same. This choice induces a homeomorphism between the associated regular orbit spaces, and an isomorphism $B(W)\simeq B(W')$ which sends braid reflections to braid reflections.
This last point gives that the actions of $B(W)$ and $B(W')$ on $\Z_\epsi$ and $k[t,t^{-1}]$ are the same and we only need to compute the associated homology for one representant of the isodiscriminantality class.

\subsection{Coxeter groups and Artin groups}\label{sec:coxeter_groups}
The first case we are going to consider is that of complexified real reflection groups. We refer to \cite[Section IV.1]{boulie} for classical results about Coxeter groups and real reflection groups.

Consider $(W,S)$ a Coxeter system of spherical type. For $s,t\in S$, we denote by $m_{s,t}$ the order of $st$ in $W$. The \nit{Artin group} associated to $(W,S)$ is defined by the following presentation
\[A(W):=\left\langle S~|~  \langle s,t\rangle^{m_{s,t}}=\langle t,s\rangle^{m_{t,s}}~\forall s\neq t\right\rangle\]
where $\langle x,y\rangle^m$ denotes the product $xyxy...$ with $m$ terms. 

It is known (see \cite{brieskorn}) that the braid group of $W$ seen as a complex reflection group is isomorphic to $A(W)$ (and this isomorphism sends the elements of $S$ to braid reflections).

The presentation of $A(W)$ also gives rise to a monoid, denoted by $M(W)$. The monoid $M(W)$ is always locally left-Gaussian, and since $(W,S)$ is of spherical type, it is Gaussian (see \cite[Example 1]{dehpar}). The monoid $M(W)$ is the \nit{Artin monoid} associated to the Coxeter system $(W,S)$. 

The homology of Artin monoids has already been studied by Salvetti (see \cite{salvetti}), and by Squier (see  \cite{squier}). The approach of the latter was then generalized in \cite[Section 4]{dehlaf} into the order complex.

As we want to use solely the Dehornoy-Lafont complex in our computations, we give its construction in the case of a Coxeter group.

If $S'$ is a subset of $S$, then the subgroup $W'$ generated by $S'$ in $W$ is also a spherical Coxeter group. The Artin monoid $M'$ associated to $(W',S')$ is the submonoid generated by $S'$ in $M$. In particular the right-lcm of the elements of $S'$ lies inside $M'$. So the atoms right-dividing this lcm are precisely the elements of $S'$. We get the following lemma:

\begin{lem} Let $(S,W)$ be a Coxeter system of spherical type, with $S=\{s_1,\ldots,s_n\}$. For $k$ be a positive integer, the $k$-cells for the Dehornoy-Lafont complex associated to $(S,W)$ are given by
\[X_k:=\left\{[s_{i_1},\ldots,s_{i_k}]~|~ 1\leqslant i_1<\ldots<i_k\leqslant n\right\}.\]
In particular the cardinality of $X_k$ is $\binom{n}{k}$ and does not depend on the choice of an order of the atoms.
\end{lem}

This is an extreme case where the tools of subsection \ref{mini} don't apply. Indeed the lcm of two distinct atoms $s,t$ is always of the form $\ell=\langle s,t\rangle^{m_{s,t}}$, with $\Aa_\ell=\{s,t\}$. This means that the two bounds of Lemma \ref{4.1} are equal in this case.

The classical Artin monoid has the advantage of providing relatively few cells. Even for the largest case (which is $E_8\simeq G_{37}$), we have at most $\binom{8}{4}=70$ cells (in rank $4$). On the other hand, the differential is often long to compute because of recursion: elements of the form $\alpha_{/A}$ may have great length, because the simple elements in the Artin monoids can have length up to 120 in the case of $E_8$.

As seen in the tables below, the case of the Artin monoids covers the exceptional groups which are complexified real reflection groups:
\[B_{23}\simeq A(H_3),~B_{28}\simeq A(F_4),~B_{30}\simeq A(H_4),~B_{35}=A(E_6),~B_{36}=A(E_7),~B_{37}\simeq A(E_8)\]
Furthermore, some exceptional groups are known to be isodiscriminantal to complexified real reflection groups (see \cite[Theorem 2.25]{orliksoldiscri} and \cite[Section 2]{bannai}). Therefore the study of Artin monoids also gives the homology of the following groups 
\[B_4,~B_5,~B_6,~B_8,~B_9,~B_{10},~B_{14},~B_{16},~B_{17},~B_{18},~B_{20},~B_{21},~B_{25},~B_{26},~B_{32}\]

\subsection{Exceptional groups of rank two}
The next case we consider is that of exceptional groups of rank two. Some of these groups are known to be isodiscriminantal to complexified real groups, which we already considered in Section \ref{sec:coxeter_groups}. This leaves the following groups:
\[B_7,~B_{11},~B_{12},~B_{13},~B_{15},~B_{19},~B_{22}.\]
Among these, $B_7$, $B_{11}$ and $B_{19}$ are isodiscriminantal (see \cite[Section 2]{bannai}). In order to study these various groups, we use \emph{ad hoc} Garside monoids, which are all detailed in \cite[Examples 11, 12, 13]{thespicantin}. Most of these monoids are \emph{circular monoids}: 
\[B_7\simeq B_{11}\simeq B_{19}=\langle a,b,c~|~ abc=bca=cab\rangle\]
\[B_{12}=\langle a,b,c~|~ abca=bcab=cabc\rangle\]
\[B_{22}=\langle a,b,c~|~ abcab=bcabc=cabca\rangle\]
These are group presentation, which can also be seen as monoid presentation. We amalgamate such a group presentation and the underlying monoid presented by the same data. In these three monoids, one can see that the atoms play a symmetric role: changing the ordering on the atoms does not affect the number of cells. Furthermore, as the lcm of two distinct atoms is always the same, we get that there are only $1$-cells and $2$-cells ($n$ cells of rank $1$ and $n-1$ cells of rank 2 if $n$ is the number of atoms). 

For $B_{13}$, we use the following monoid:
\[B_{13}=\langle a,b,c~|~ acabc=bcaba,~bcab=cabc,~cabca=abcab\rangle.\]
Using the notations of Section \ref{mini}, we have
\[\ell_1:=b\vee c=bcab,~ \ell_2:=a\vee b=a\vee c=abcab\]
For $\ell_1$, we have $\Aa_{\ell_1}=\{b,c\}$ and $n(b,\ell_1)=n(c,\ell_1)=1$, so there is no use in setting either $b<c$ or $c<b$. For $\ell_2$, we have $\Aa_{\ell_2}=\{a,b,c\}$ and $n(a,\ell_2)=2,n(b,\ell_2)=n(c,\ell_2)=1$. So by considering the order $c<a<b$, we get one $1$-cell, two $2$-cells, and zero $3$-cells. If we instead consider the order $a<b<c$, we get three $2$-cells instead of two. 

Lastly, for $B_{15}$, we use the monoid
\[B_{15}=\left\langle a,b,c~|~ abc=bca,~cabcb=abcbc\right\rangle \]
We have
\[\ell_1=a\vee c=abc,~~\ell_2=b\vee c=b\vee a=abcbc\]
For $\ell_1$, we have $\Aa_{\ell_1}=\{a,c\}$ and $n(a,\ell_1)=n(c,\ell_1)=1$, so there is again no use in setting a priori that either $a<c$ or $c<a$. For $\ell_2$, we have $\Aa_{\ell_2}=\{a,b,c\}$ and $n(a,\ell_2)=n(c,\ell_2)=1$, $n(b,\ell_2)=2$. So we consider the order $c<a<b$ in order to get as few cells as possible.

\subsection{Well-generated exceptional groups}
At this point, we still have six groups to consider, and five of them are ``well-generated'' in the sense of \cite[Section 2]{beskpi1}:  
\[B_{24},~B_{27},~B_{29},~B_{33},~B_{34}.\]
The monoid we use for these groups is the \nit{dual braid monoid} (see \cite[Section 8]{beskpi1}). The main problem with these monoids is that they have many atoms, and so they give rise to relatively big complexes. This is where the methods of Section \ref{mini} are most useful. The bounds of Lemma \ref{4.1} for the number of $2$-cells are respectively given by

\spa
\begin{center}
\begin{tabular}{|c|ccccc|}
\hline & $B_{24}$ & $B_{27}$ & $B_{29}$ & $B_{33}$ & $B_{34}$\\
\hline Lower bound & 38 & 60 & 120 & 213 & 630 \\
Upper bound & 40 & 65 & 158 & 302 & 1071\\
\hline 
\end{tabular}
\end{center}

Applying the method of Section \ref{mini} to get a convenient order, we see in Table \ref{tab:comparedsizemarin} that the lower bounds are not always reached, but we obtain complexes which are quite smaller than the ones in \cite[Table 1]{homcomp2}, especially for large groups.

\begin{center}
\begin{table}[h]
\begin{tabular}{|c|lccccccc|}
\hline
& & 0-cells&1-cells & 2-cells&3-cells&4-cells & 5-cells&6-cells\\
\hline $B_{24}$ & \cite{homcomp2} & 1 & 14 & 38 & 25& & & \\
& Optimized & 1 & 14 & 38 & 25 & & & \\
\hline $B_{27}$ & \cite{homcomp2} & 1 & 20 & 62 & 43 & & & \\
& Optimized & 1 & 20 & 60 & 41 & & & \\
\hline $B_{29}$ & \cite{homcomp2} & 1 & 25 & 127 & 207 & 108 & & \\
& Optimized & 1 & 25 & 125 & 209 & 108 & & \\
\hline $B_{33}$ & \cite{homcomp2} & 1 & 30 & 226 & 638 & 740 & 299 & \\
& Optimized & 1 & 30 & 223 & 616 & 705 & 283 & \\
\hline $B_{34}$ & \cite{homcomp2} & 1 & 56 & 711 & 3448 & 7520 &7414 & 2686 \\
& Optimized & 1 & 56 & 646 & 2839 & 5691 & 5255 & 1812\\
\hline
\end{tabular}
\caption{Compared size of the Dehornoy-Lafont complexes}\label{tab:comparedsizemarin}
\end{table}
\end{center}

Sadly, although we obtain a smaller complex for $B_{34}$, it is not small enough to obtain the results that were missing in \cite{homcomp1} regarding $H_*(B_{34},\Q[t,t^{-1}])$. However, we were able to compute $H_*(B_{34},k[t,t^{-1}])$ where $k$ ranges among some finite fields. These computations give us a reasonable conjecture regarding $H_*(B_{34},\Q[t,t^{-1}])$.

\subsection{The Borchardt braid group $B_{31}$} The last exceptional group to consider is $B_{31}$. Although this groups does not appear (to our knowledge) as a group of fraction of some Gaussian monoid, it is equivalent to the enveloping groupoid of some Garside category. Indeed the complex reflection group $G_{31}$ appears as the centralizer of some regular element (in the sense of Springer) in the Coxeter group $E_8$. Following the work of Bessis in \cite[Section 11]{beskpi1}, this description gives rise to a Garside category $\Cc_{31}$, whose envelopping groupoid $\Bb_{31}$ is equivalent to $B_{31}$. A detailed description of this category can be found in \cite{springercat}.

This category admits $88$ objects, $660$ atoms, and a total of $2603$ simple morphisms (excluding the identities). The first possible approach used to study the homology of this category (see \cite[Section 5.3]{homcomp1}) was to construct the Charney-Meier-Whittlesey complex for this category (as defined in \cite[Section 7]{besgar}). Sadly this complex is too large to be dealt with without a strong computational power. 

We give in Table \ref{tab:comparedsizeborchardt} the size of the Charney-Meier-Whittlesey complex for $\Cc_{31}$ compared to the size of the Dehornoy-Lafont complex (note that in this case, the bounds given by Lemma \ref{4.1} are 1655 and 1845, respectively):

\begin{table}[h]
\begin{center}
\begin{tabular}{|c|ccccc|}
\hline & 0-cells & 1-cells & 2-cells& 3-cells & 4-cells \\
\hline CMW & 88 & 2603 & 11065 & 15300 & 6750\\
 DL & 88 & 660 & 1665 & 1735 & 642\\ \hline
\end{tabular}
\end{center}
\caption{Compared size of the complexes for $\Cc_{31}$}\label{tab:comparedsizeborchardt}
\end{table}

Let $M$ be one of the $B_{31}$ modules we are considering. We first extend $M$ to a $\Bb_{31}$ module, using the construction of Section \ref{cgg}. We then restrict this module to the category $\Cc_{31}$. In the case of $k[t,t^{-1}]$, the matrix we obtain may contain coefficients in $k[t,t^{-1}]\setminus k[t]$, as opposed to the case of a monoid, in which the action gives rise to matrices in $k[t]$. In theory this is not a problem since $k[t,t^{-1}]$ is a principal ideal domain. But in practice, it is far simpler to look for the Smith normal form of a matrix in $k[t]$. To avoid this issue, we multiply our matrices by a big enough power of $t$ (which is an invertible element in $k[t,t^{-1}]$) to obtain matrices in $k[t]$. We only then need to divide the elementary divisors we obtain by powers of $t$ if need be.

\subsection{Computational results} We use the notation $\Z_n$ for $\Z/n\Z$. The computations for the complexes and the differentials were made on the {\tt CHEVIE} package for {\tt GAP3} (\cite{gap3}). The computations of the Smith normal forms were made using the softwares {\tt Macaulay2} (\cite{macaulay2}) and {\tt MAGMA} (\cite{magma}).

For each row, we indicate the representative of the isodiscriminantality class of which we computed the homology.

The results in Table \ref{tab3} are not new. The case of complexified real reflection groups is already known from \cite{salvetti}; the case of complex reflection groups which are not isodiscriminantal to groups in the infinite series is given in \cite{homcomp1} and \cite{homcomp2}. We reproduce their results here for the convenience of the reader

\begin{table}[h]
\begin{center}
\begin{tabular}{|r|ccccccccc|}
\hline & $H_0$ & $H_1$ & $H_2$ & $H_3$ & $H_4$ & $H_5$ & $H_6$ & $H_7$ & $H_8$\\
 \hline $A_2\sim G_4$, $G_{8}$, $G_{16}$ & $\Z$ & $\Z$ & 0 &  &  & & & & \\
 $I_2(4) \sim G_5$, $G_{10}$, $G_{18}$ & $\Z$ & $\Z^2$ & $\Z$ &  &  & & & & \\
  $I_2(6) \sim G_6$, $G_9$, $G_{17}$ & $\Z$ & $\Z^2$ & $\Z$ &  &  & & & & \\
  $G_7$, $G_{11}$, $G_{19}$ & $\Z$ & $\Z^3$ & $\Z^2$ &  &  & & & & \\
  $G_{12}$ & $\Z$ & $\Z$ & 0 &  &  & & & & \\
  $G_{13}$ & $\Z$ & $\Z^2$ & $\Z$ & &  & & & & \\
  $I_2(8) \sim G_{14}$ & $\Z$ & $\Z^2$ & $\Z$ &  &  & & & & \\
  $G_{15}$ & $\Z$ & $\Z^3$ & $\Z^2$ &  &  & & & & \\
  $I_2(5) \sim G_{20}$ & $\Z$ & $\Z$ & $0$ &  &  & & & & \\
  $I_2(10) \sim G_{21}$ & $\Z$ & $\Z^2$ & $\Z$ &  &  & & & & \\
  $G_{22}$ & $\Z$ & $\Z$ & $0$ &  &  & & & & \\
  $G_{23}=H_3$ & $\Z$ & $\Z$ & $\Z$ & $\Z$ &  & & & & \\
  $G_{24}$ & $\Z$ & $\Z$ & $\Z$ & $\Z$ &  & & & & \\
  $A_3\sim G_{25}$ & $\Z$ & $\Z$ & $\Z_2$ & $0$ &  & & & & \\
  $B_3\sim G_{26}$ & $\Z$ & $\Z^2$ & $\Z^2$ & $\Z$ &  & & & & \\
  $G_{27}$ & $\Z$ & $\Z$ & $\Z_3\times \Z$ & $\Z$ &  & & & & \\
  $G_{28}= F_4$ & $\Z$ & $\Z^2$ & $\Z^2$ & $\Z^2$ & $\Z$ & & & & \\
  $G_{29}$ & $\Z$ & $\Z$ & $\Z_2\times \Z_4$ & $\Z_2\times \Z$ & $\Z$  & & & & \\
  $G_{30}= H_4$ & $\Z$ & $\Z$ & $\Z_2$ & $\Z$ & $\Z$ & & & & \\
  $G_{31}$ & $\Z$ & $\Z$ & $\Z_6$ & $\Z$ & $\Z$ & & & & \\
  $A_4\sim G_{32}$ & $\Z$ & $\Z$ & $\Z_2$ & $0$ & $\Z$ & & & & \\
  $G_{33}$ & $\Z$ & $\Z$ & $\Z_6$ & $\Z_6$ & $\Z$ & $\Z$ & & & \\
  $G_{34}$ & $\Z$ & $\Z$ & $\Z_6$ & $\Z_6$ & $\Z_3^2\times \Z_6$ & $\Z_3^2\times \Z$& $\Z$ & & \\
  $G_{35}=E_6$ & $\Z$ & $\Z$ & $\Z_2$ & $\Z_2$ & $\Z_6$ & $\Z_3$ & $0$ & & \\
  $G_{36}=E_7$ & $\Z$ & $\Z$ & $\Z_2$ & $\Z_2^2$ & $\Z_6^2$ & $\Z_3\times \Z_6$ & $\Z$ & $\Z$& \\
  $G_{37}=E_8$ & $\Z$ & $\Z$ & $\Z_2$ & $\Z_2$ & $\Z_2\times \Z_6$ & $\Z_3\times \Z_6$ &$\Z_2\times \Z_6$ & $\Z$& $\Z$\\
 \hline 
\end{tabular}
\end{center}
\caption{Homology of exceptional braid groups in $\Z$ (after Salvetti, Callegaro and Marin)\label{tab3}}
\end{table}

\newpage
In the same vein, Table \ref{tab4} gives the homology of exceptional braid groups with coefficients in the sign representation. This homology has already been studied in \cite{salvetti} for complexified real reflection groups and in \cite{homcomp2} for the exceptional groups  $B_{12}, B_{13}, B_{22}, B_{24}, B_{27},$ $B_{29},$  $B_{33}, B_{34}$. The first homology groups was studied for all complex reflection group in \cite[Section 7.2]{homcomp1}. We restate the results of \cite[Table 2]{salvetti} and \cite[Table 3]{homcomp2} among the results for other exceptional groups (we frame the results which are new).


\begin{table}[h]
\begin{center}
\begin{tabular}{|r|ccccccccc|}
\hline & $H_0$ & $H_1$ & $H_2$ & $H_3$ & $H_4$ & $H_5$ & $H_6$ & $H_7$ & $H_8$\\
 \hline $A_2\sim G_4$, $G_{8}$, $G_{16}$ & $\Z_2$ & $\Z_3$ & 0 &  &  & & & & \\
 $I_2(4)\sim G_5$, $G_{10}$, $G_{18}$ & $\Z_2$ & $\Z_4$ & 0 &  &  & & & & \\
  $I_2(6) \sim G_6$, $G_9$, $G_{17}$ & $\Z_2$ & $\Z_6$ & 0 &  &  & & & & \\
  $G_7$, $G_{11}$, $G_{19}$ & $\Z_2$ & $\Z_2^2$ & \fbox{0} &  &  & & & & \\
  $G_{12}$ & $\Z_2$ & $\Z_3$ & $0$ &  &  & & & & \\
  $G_{13}$ & $\Z_2$ & $\Z_2$ & $0$ &  &  & & & & \\
  $I_2(8)\sim G_{14}$ & $\Z_2$ & $\Z_8$ & $0$ &  &  & & & & \\
  $G_{15}$ & $\Z_2$ & $\Z_2^2$ & \fbox{$0$} &  &  & & & & \\
  $I_2(5)\sim G_{20}$ & $\Z_2$ & $\Z_5$ & $0$ &  &  & & & & \\
  $I_2(10) \sim G_{21}$ & $\Z_2$ & $\Z_{10}$ & $0$ &  &  & & & & \\
  $G_{22}$ & $\Z_2$ & $0$ & $0$ &  &  & & & & \\
  $G_{23}=H_3$ & $\Z_2$ & $0$ & $\Z_2$ & 0 &  & & & & \\
  $G_{24}$ & $\Z_2$ & 0 & $\Z_2$ & 0 &  & & & & \\
  $A_3\sim G_{25}$ & $\Z_2$ & $\Z_3$ & $\Z_2$ & $0$ &  & & & & \\
  $B_3\sim G_{26}$ & $\Z_2$ & $\Z_2$ & $\Z_2$ & $0$ &  & & & & \\
  $G_{27}$ & $\Z_2$ & 0 & $\Z_2$ & 0 &  & & & & \\
  $G_{28}=F_4$ & $\Z_2$ & $\Z_2$ & $\Z_6$ & $\Z_{24}$ & 0 & & & & \\
  $G_{29}$ & $\Z_2$ & 0 & $\Z_2\times\Z_4$ & $\Z_2\times \Z_{40}$ & 0 & & & & \\
  $G_{30}=H_4$ & $\Z_2$ & $0$ & $\Z_2$ & $\Z_{120}$ & $0$ & & & & \\
  $G_{31}$ & $\Z_2$ & 0 & \fbox{$\Z_6$} & \fbox{$\Z_{20}$} & \fbox{$0$} & & & & \\
  $A_4 \sim G_{32}$ & $\Z_2$ & $0$ & $\Z_2$ & $\Z_5$ & $0$ & & & & \\
  $G_{33}$ & $\Z_2$ & 0 & $\Z_2$ & $\Z_2$ & $\Z_2$ & 0 & & & \\
  $G_{34}$ & $\Z_2$ & 0 & $\Z_6$ & $\Z_2$ &$\Z_6$  & $\Z_{252}$& 0 & & \\
  $G_{35}=E_6$ & $\Z_2$ & $0$ & $\Z_2$ & $\Z_2$ & $\Z_2$ &$\Z_9$ &$0$ & & \\
  $G_{36}=E_7$ & $\Z_2$ & $0$ & $\Z_2$ & $\Z_2^2$ & $\Z_2^2$ & $\Z_2$ &$\Z_2$ & $\Z$ & \\
  $G_{37}=E_8$ & $\Z_2$ & $0$ & $\Z_2$ & $\Z_2$ & $\Z_2^2$ & $\Z_2$ & $\Z_2^2$ &$\Z_{240}$ &0 \\
 \hline 
\end{tabular}
\end{center}
\caption{Homology of exceptional braid groups in $\Z_\epsi$ (framed results are new).}\label{tab4}
\end{table}

Lastly, we give in Table \ref{tab5} the homology with coefficients in the representation $k[t,t^{-1}]$. In the case $k=\Q$, this was already studied in \cite{salvetti} for complexified real exceptional reflection groups; in \cite{conciniprocesisalvetti} for real reflection groups of type $A$; in \cite{arithmeticpropofcohomofartingroups} for real reflection groups of type $B$; and in \cite{homcomp2} for the exceptional groups $B_{12}, B_{13}, B_{22}, B_{24}, B_{27}, B_{29}, B_{33},$ $B_{34}$, although the results are incomplete for this last group.  We let $\Phi_n \in \Z[t]$ denote the $n$-th cyclotomic polynomial. In the Table, for each $P\in \Q[t,t^{-1}]$, the presence of $P$ in the Table symbolizes the $\Q[t,t^{-1}]$ module $\Q[t,t^{-1}]/(P)$, and $\Q$ is a shortcut for $\Q[t,t^{-1}]/(t-1)$.

Note that our apparent results regarding the groups studied in \cite{homcomp2} differ from \cite[Table 4]{homcomp2} because of a slight mistake in the latter: for all groups except $B_{13}$, the results should be ``shifted to the left'', for instance $\Phi_6\oplus \Phi_{12}$ is not $H_2(B_{12},\Q[t,t^{-1}])$, but rather $H_1(B_{12},\Q[t,t^{-1}])$.


\newpage
\begin{small}
\begin{landscape}
\begin{table}[h]
\begin{center}
\begin{tabular}{|r|ccccccccc|}
\hline & $H_0$ & $H_1$ & $H_2$ & $H_3$ & $H_4$ & $H_5$ & $H_6$ & $H_7$ & $H_8$\\
 \hline $A_2\sim G_4$, $G_{8}$, $G_{16}$ & $\Q$ & $\Phi_6$ & 0 &  &  & & & & \\
 $I_2(4) \sim G_5$, $G_{10}$, $G_{18}$ & $\Q$ & $\Phi_1\Phi_4$ & 0 &  &  & & & & \\
  $I_2(6)\sim G_6$, $G_9$, $G_{17}$ & $\Q$ & $\frac{t^6-1}{t+1}$ & 0 &  &  & & & & \\
  $G_7$, $G_{11}$, $G_{19}$ & $\Q$ & \fbox{$\Q\oplus (t^3-1)$} & \fbox{0} &  &  & & & & \\
  $G_{12}$ & $\Q$ & $\Phi_6\Phi_{12}$ & 0 &  &  & & & & \\
  $G_{13}$ & $\Q$ & $\Phi_1\Phi_9$ & $0$ &  &  & & & & \\
  $I_2(8) \sim G_{14}$ & $\Q$ & $\frac{t^8-1}{t+1}$ & 0 &  &  & & & & \\
  $G_{15}$ & $\Q$ & \fbox{$\Q\oplus t^5-1$} & \fbox{0} &  &  & & & & \\
  $I_2(5) \sim G_{20}$ & $\Q$ & $\Phi_{10}$ & 0 &  &  & & & & \\
  $I_2(10) \sim G_{21}$ & $\Q$ & $\frac{t^{10}-1}{t+1}$ & $0$ &  &  & & & & \\
  $G_{22}$ & $\Q$ & $\Phi_{15}$ & 0 &  &  & & & & \\
  $G_{23}=H_3$ & $\Q$ & $0$ & $\frac{t^5-1}{t-1} \Phi_3$ & 0 &  & & & & \\
  $G_{24}$ & $\Q$ & 0 & $\Phi_1\Phi_3\Phi_7$ & 0 &  & & & & \\
  $A_3\sim G_{25}$ & $\Q$ & $\Phi_6$ & $\Phi_4$ & $0$ &  & & & & \\
  $B_3\sim G_{26}$ & $\Q$ & $\Q$ & $t^3-1$ & $0$ &  & & & & \\
  $G_{27}$ & $\Q$ & 0 & $(t^{15}-1)\oplus \Phi_3$ & 0 &  & & & & \\
  $G_{28}=F_4$ & $\Q$ & $\Q$ & $\frac{t^6-1}{t+1}$ & $\frac{t^{12}-1}{t+1}\Phi_8$ & 0 & & & & \\
  $G_{29}$ & $\Q$ & $0$ & $\Phi_4\oplus \Phi_4$ & $\frac{t^{20}-1}{t+1}\oplus \Phi_4$ & 0 & & & & \\
  $G_{30}=H_4$ & $\Q$ & $0$ & $0$ & $\frac{t^{30}-1}{t+1}\Phi_4\Phi_{12}\Phi_{20}$ & 0 & & & & \\
  $G_{31}$ & $\Q$ & \fbox{0} & \fbox{$\Phi_6$} & \fbox{$\frac{t^{10}-1}{t+1} \Phi_{15}$} & \fbox{$0$} & & & & \\
  $A_4\sim G_{32}$ & $\Q$ & 0 & $\Phi_4$ & $\Phi_{10}$& 0 & & & &\\
  $G_{33}$ & $\Q$ & 0 & 0 & 0 & $(t^9-1)\Phi_5$ & 0 & & & \\
  $G_{34}$ & $\Q$ & 0 & $\Phi_6$ & ? & ? & ? & ? & & \\
  $G_{35}=E_6$ & $\Q$ & 0  & 0 & 0 & $\Phi_3\Phi_8$ & $\frac{t^{12}-1}{t^4-1}\Phi_{18}$& 0 & & \\
  $G_{36}=E_7$ & $\Q$ & $0$ & $0$ & $0$ & $\Phi_3$ & $\Phi_3$ & $(t^9-1)\Phi_7$& $0$& \\
  $G_{37}=E_8$ & $\Q$ & 0 & 0 & 0 & $\Phi_4$ & 0&$\Phi_{8}\Phi_{12}$ & $\frac{t^{30}-1}{t+1} \frac{t^{24}-1}{t^6-1} \Phi_{20}$&0 \\
 \hline 
\end{tabular}
\end{center}
\caption{Homology of exceptional braid groups in $\Q[t,t^{-1}]$ (framed results are new).}\label{tab5}
\end{table}

\end{landscape}
\end{small}
\newpage 

We finish with the case of a finite field $k$. Following \cite{homcomp2}, we restrict our attention to the case $k=\F_p$ with $p\in \{2,3,5,7\}$. We denote by $\phi_{n,k}$ the $n$-th cyclotomic polynomial with coefficients in $k$. As $\Phi_{n}:=\phi_{n,\Q}$ lies in $\Z[X]$, we can consider its image in $\F_p[X]$ for some prime $p$. We also denote this polynomial by $\Phi_{n}$. It is well known that $\Phi_{n}=\phi_{n,\F_p}$ mod $p$ if $p$ does not divide $n$. Furthermore, if $p$ does not divide $n$, then we have 
\[\forall r>0,~\Phi_{np^r}\equiv (\Phi_{n})^{p^r-p^{r-1}}=(\phi_{n,\F_p})^{p^r-p^{r-1}}\text{ mod }p\]
as stated in \cite{cyclo}. Most of the time, the homology $H_*(B_i,\F_p[t,t^{-1}])$ is given by the same polynomials as $H_*(B_i,\Q[t,t^{-1}])$, reduced mod $p$. For instance we have, in the case of $B_{12}$: 
\[H_1(B_{12},\F_2[t,t^{-1}])=\F_2[t,t^{-1}]/(P)\]
Where
\[P(X)=X^6-X^5+X^3-X+1=(X^2+X+1)^3= \Phi_{6}(X)\Phi_{12}(X)\equiv \phi_{3,\F_2}(X)^3\text{ mod }2\]
We list here the cases where $H_*(B_i,\F_p[t,t^{-1}])$ is not given by the reduction modulo $p$ of the polynomials giving $H_*(B_i,\Q[t,t^{-1}])$. We first consider the case $W\neq G_{34}$. The results for $W\in \{G_{24},G_{29},G_{33}\}$ already appear in \cite{homcomp2}.

\begin{enumerate}[$\bullet$]
\item When $W=G_{29}$, we have \begin{enumerate}[-]
\item $H_3(B_{29},\F_2[t,t^{-1}])=(t+1)^3\oplus \Phi_4$.
\item $H_4(B_{29},\F_2[t,t^{-1}])=(t^{20}-1)\oplus \Phi_4$.
\end{enumerate}
\item When $W=G_{30}$, we have $H_3(B_{30},\F_2[t,t^{-1}])=(t^{30}-1)\Phi_4\Phi_{12}\Phi_{20}$.
\item When $W=G_{31}$, we have \begin{enumerate}[-]
\item $H_2(B_{31},\F_2[t,t^{-1}])=\Phi_1\Phi_6$ and $H_2(B_{31},\F_3[t,t^{-1}])=\Phi_1\Phi_6$.
\item $H_3(B_{31},\F_2[t,t^{-1}])=(t^{10}-1)\Phi_{15}$ and $H_3(B_{31},\F_3[t,t^{-1}])=\frac{t-1}{t+1}(t^{10}-1)\Phi_{15}$.
\end{enumerate}
\item When $W=G_{33}$, we have \begin{enumerate}[-]
\item $H_{2}(B_{33},\F_2[t,t^{-1}])=\Phi_1$ and $H_2(B_{33},\F_3[t,t^{-1}])=\Phi_1$.
\item $H_{3}(B_{33},\F_2[t,t^{-1}])=\Phi_1$ and $H_3(B_{33},\F_3[t,t^{-1}])=\Phi_1$.
\end{enumerate}
\item When $W=G_{35}$, we have $H_2(B_{35},\F_2[t,t^{-1}])=\Phi_1$ and $H_3(B_{35},\F_2[t,t^{-1}])=\Phi_1$.
\item When $W=G_{36}$, we have \begin{enumerate}[-]
\item $H_2(B_{36},\F_2[t,t^{-1}])=\Phi_1$.
\item $H_3(B_{36},\F_2[t,t^{-1}])=\Phi_1\oplus \Phi_1$.
\item $H_4(B_{36},\F_2[t,t^{-1}])=\Phi_1\oplus (t^3+1)$ and $H_4(B_{36},\F_3[t,t^{-1}])=t^3-1$.
\item $H_5(B_{36},\F_2[t,t^{-1}])=t^3+1$, and $H_5(B_{36},\F_3[t,t^{-1}])=t^3-1$.
\end{enumerate}
\item When $W=G_{37}$, we have \begin{enumerate}[-]
\item $H_2(B_{37},\F_2[t,t^{-1}])=\Phi_1$.
\item $H_3(B_{37},\F_2[t,t^{-1}])=\Phi_1$.
\item $H_4(B_{37},\F_2[t,t^{-1}])=\Phi_1\Phi_4$ and $H_4(B_{37},\F_3[t,t^{-1}])=\Phi_1\Phi_4$.
\item $H_5(B_{37},\F_2[t,t^{-1}])=\Phi_1$ and $H_5(B_{37},\F_3[t,t^{-1}])=\Phi_1\oplus \Phi_1$.
\item $H_6(B_{37},\F_2[t,t^{-1}])=\Phi_1\Phi_8\Phi_{12}$ and $H_6(B_{37},\F_3[t,t^{-1}])=\Phi_1\Phi_8\Phi_{12}$.
\item $H_7(B_{37},\F_2[t,t^{-1}])=\frac{(t^{30}-1)(t^{24}-1)}{(t^6-1)}\Phi_{20}$.
\end{enumerate}
\end{enumerate}

We notice that, in all of these cases, the homology $H_*(B_i,\F_p[t,t^{-1}])$ for $p\in \{5,7\}$ is given by the reduction modulo $p$ of the polynomials giving $H_*(B_i,\Q[t,t^{-1}])$. This allows us to give a conjecture about the values of $H_*(B_{34},\Q[t,t^{-1}])$ which we were not able to compute.

The homology $H_*(B_{34},\F_p[t,t^{-1}])$ for $p\in \{2,3\}$ was computed in \cite{homcomp2}. The results are in Table \ref{tab:B34F2F3}.

\begin{table}[h]
\begin{center}
\begin{tabular}{|r|ccccccc|}
\hline $G_{34}$& $H_0$ & $H_1$ & $H_2$ & $H_3$ & $H_4$ & $H_5$ & $H_6$ \\
\hline $\F_2[t,t^{-1}]$  & $\F_2$ & 0 & $(t^3-1)$ & $\F_2$ & $\Phi_{3}\oplus \Phi_3 \oplus (t^3-1)$ & $ {(t^{42}-1)}\oplus \Phi_3\oplus \Phi_3$ & 0 \\
 $\F_3[t,t^{-1}]$  & $\F_3$ & 0 & $\F_3\oplus \Phi_6$ & $\F_3$ & $\Phi_{3}\oplus \Phi_3 \oplus \Phi_3\oplus \Phi_2$ & $ {(t^{42}-1)}\oplus \Phi_3\oplus \Phi_3$ & 0 \\
  \hline 
\end{tabular}
\end{center}
\caption{Homology of $B_{34}$ with coefficients in $\F_2[t,t^{-1}]$ and $\F_3[t,t^{-1}]$ (after Marin).}\label{tab:B34F2F3}
\end{table}

We were able to compute the homology $H_*(B_{34},\F_p[t,t^{-1}])$ for all primes $5\leqslant p\leqslant 97$. For each of these cases, we see in Table \ref{tab:B34Fp} that the homology is given by the same polynomials:
\begin{table}[h]
\begin{center}
\begin{tabular}{|r|ccccccc|}
\hline $G_{34}$& $H_0$ & $H_1$ & $H_2$ & $H_3$ & $H_4$ & $H_5$ & $H_6$ \\
\hline $\F_p[t,t^{-1}]$  & $\F_p$ & 0 & $\Phi_6$ & $0$ & $\Phi_{3}\oplus \Phi_3 \oplus \Phi_3$ & $ \frac{(t^{42}-1)}{t+1}\oplus \Phi_3\oplus \Phi_3$ & 0 \\
  \hline 
\end{tabular}
\end{center}
\caption{Homology of $B_{34}$ with coefficients in $\F_p[t,t^{-1}]$ (for primes $p$ between $5$ and $97$).}\label{tab:B34Fp}
\end{table}

We obtain a conjecture about $H_*(B_{34},\Q[t,t^{-1}])$, which we state in Table \ref{tab:conj}.

\begin{table}[h]
\begin{center}
\begin{tabular}{|r|cccc|}
\hline $G_{34}$& $H_3$ & $H_4$ & $H_5$ & $H_6$ \\
\hline $\Q[t,t^{-1}]$& 0 & $\Phi_{3}\oplus\Phi_{3}\oplus \Phi_{3}$ & $\frac{t^{42}-1}{t+1}\oplus \Phi_{3}\oplus \Phi_3$  & 0 \\
  \hline 
\end{tabular}
\end{center}
\caption{Conjectural homology of $B_{34}$ with coefficients in $\Q[t,t^{-1}]$.}\label{tab:conj}
\end{table}

\printbibliography
\end{document}